\documentclass[reqno]{amsart}
\usepackage{graphicx} 
\usepackage{amsmath,amsthm,amssymb,amsfonts} 
\usepackage{mathtools}
\usepackage{mathrsfs}
\usepackage{stmaryrd}
\usepackage{enumerate} 
\usepackage{tikz-cd}
\usepackage{comment}
\usepackage{soul, xcolor}
\usepackage[left=1.2in, right=1.2in, top=1.3in, bottom=1.3in]{geometry}
\usepackage{hyperref}
\hypersetup{colorlinks=true, citecolor=blue}
\usepackage[toc,page]{appendix}
\usepackage{MnSymbol}
\usepackage[
style=alphabetic
]{biblatex}
\newtheorem{theorem}{Theorem}[section]
\newtheorem{lemma}[theorem]{Lemma}
\newtheorem{prop}[theorem]{Proposition}
\newtheorem{cor}[theorem]{Corollary}

\theoremstyle{definition}
\newtheorem{definition}[theorem]{Definition}

\newtheorem{remark}[theorem]{Remark}
\newtheorem{assumption}[theorem]{Assumption}

\newcommand{\A}[2]{\mathbb{A}_{#1}^{#2}}
\newcommand{\Q}{\mathbb{Q}}
\newcommand{\Z}{\mathbb{Z}}
\newcommand{\N}{\mathbb{N}}
\newcommand{\C}{\mathbb{C}}
\newcommand{\R}{\mathbb{R}}
\newcommand{\F}[1]{\mathbb{F}_{#1}}
\newcommand{\GL}[1]{\mathrm{GL}_{#1}}

\DeclareMathOperator{\AJ}{AJ}
\DeclareMathOperator{\CH}{CH}
\DeclareMathOperator{\cl}{cl}
\DeclareMathOperator{\dR}{dR}
\DeclareMathOperator{\Fil}{Fil}
\DeclareMathOperator{\fp}{fp}
\DeclareMathOperator{\Frob}{Frob}
\DeclareMathOperator{\Hom}{Hom}
\DeclareMathOperator{\Isom}{Isom}
\DeclareMathOperator{\pr}{pr}
\DeclareMathOperator{\Res}{Res}
\DeclareMathOperator{\Spa}{Spa}
\DeclareMathOperator{\Spec}{Spec}
\DeclareMathOperator{\Spf}{Spf}
\DeclareMathOperator{\Sp}{Sp}

\DeclareMathOperator{\Sym}{Sym}
\DeclareMathOperator{\syn}{syn}
\DeclareMathOperator{\tr}{tr}
\DeclareMathOperator\can{can}
\DeclareMathOperator{\Gal}{Gal}

\title[$p$-adic Gross-Zagier formula for twisted triple product $p$-adic $L$-functions]{A $p$-adic Gross--Zagier formula for twisted triple product $p$-adic $L$-functions attached to finite slope families}
\author{Ting-Han Huang and Ananyo Kazi}
\date{\today}

\begin{document}

\begin{abstract}
    Our main objective in the present paper is to generalise the work of Blanco-Chac\'{o}n and Fornea on the $p$-adic Gross--Zagier formula for twisted triple product $p$-aidc $L$-function.
    We extend their main result to the case of finite slope families of Hilbert modular forms and also allow the prime $p$ to be inert in the real quadratic field $L$.
\end{abstract}

\maketitle

\section{Introduction}
The goal of this article is to generalise the $p$-adic Gross--Zagier formula for twisted triple product $p$-adic $L$-functions, proved by Blanco-Chac\'{o}n and Fornea in \cite{blanco-chacón_fornea_2020} for ordinary families, to the case of finite slope families of Hilbert modular forms, under the assumption that $p$ is unramified in the quadratic totally real field. 
In particular, we allow $p$ to be inert, which is not covered in \cite{blanco-chacón_fornea_2020}.
The formula, in short, relates special values of the $p$-adic $L$-function $\mathscr{L}_p(\omega_\mathbf{g}, \omega_\mathbf{f})$ in the balanced region, to images of generalized \textit{Hirzebruch--Zagier} cycles under the syntomic Abel--Jacobi map.

This idea can be traced back to Darmon and Rotger \cite{DR}, where they considered triple product $p$-adic $L$-functions attached to ordinary families of elliptic modular forms. 
The $p$-adic Gross--Zagier formula serves as an important tool in their sequel paper \cite{DR2} to prove the implication `analytic rank $= 0 \Rightarrow$ algebraic rank $= 0$' in a special case of the equivariant BSD-conjecture.
Based on \cite{blanco-chacón_fornea_2020}, Fornea and Jin also proved an analogous result under another setting in their recent work \cite{Fornea_Jin_2024}.

The generalization to the finite slope case requires two main ingredients.
First, one needs a definition of the twisted triple product $p$-adic $L$-function $\mathscr{L}_p(\omega_\mathbf{g}, \omega_\mathbf{f})$ attached to finite slope families of Hilbert modular forms. 
This is achieved in the thesis of the second author \cite{twisted_tripleL_finite_slope}, which extends the construction of \cite{andreatta2021triple}.
Second, one needs to compute syntomic Abel--Jacobi maps and specialisations of the $p$-adic $L$-function in the finite slope case.
This technique is developed in the thesis of the first author \cite{tripleL_and_GZ_formula_finite_slope}, in which he generalised the result of \cite{DR}. 

On the $L$-function side, suppose $L/F$ is a quadratic extension of totally real fields, and $g$ (resp. $f$) is a $\GL{2,L}$ (resp. $\GL{2,F}$) Hilbert modular form of weight $(v,n) \in \Z[\Sigma_L] \times \Z$ (resp. of weight $(w,m) \in \Z[\Sigma_F] \times \Z$). Here $\Sigma_L$ is the set of archimedean embeddings of $L$, and similarly for $\Sigma_F$.  Assume that the weights are $F$-dominated, i.e. there exists $r \in \N[\Sigma_L]$ such that $(v+r,n)_{|F} = (w,m)$, and the central character of the unitary automorphic representation $\Pi$ attached to $g$ and $f$ is trivial when restricted to $\A{F}{\times}$. Then Ichino's formula relates the central critical value $L(1/2, \Pi, \mathrm{R})$ of the twisted triple product $L$-function associated with $\Pi$ to a period integral which can be realised up to certain complex periods as the following Petersson inner product (cf. Proposition \ref{P2010})
\[
L(1/2, \Pi, \mathrm{R}) \overset{\cdot}{=}  \left(\frac{\langle \zeta^*(\delta^r g), f^*\rangle}{\langle f^*, f^*\rangle}\right)^2.
\]
Here $f^*$ is the dual modular form whose Fourier coefficients are the complex conjugates of those of $f$, and $\delta$ is the Maass--Shimura differential operator. The right hand side of the above equation is algebraic, and one hopes to define a $p$-adic $L$-function by $p$-adically interpolating such values as $g, f$ vary.
In particular, for finite slope Hilbert modular forms, one needs a theory of Coleman families of Hilbert modular forms, as developed in \cite{andreatta2016adic} for example.
But this is not sufficient since the $p$-adic analogue of the Maass--Shimura operator does not preserve overconvergence, and one has to iterate the full Gauss--Manin connection instead. For that reason one has to develop a theory of \emph{nearly overconvergent} Hilbert modular forms, as has been developed in the thesis of the second author. Section 2 of this article recalls this construction. The main result of \S\ref{section: L-function} which was proved first in \cite{twisted_tripleL_finite_slope}, is the construction of a $p$-adic $L$-function satisfying the following interpolation property, with the Euler factors defined as in Theorem \ref{Theorem: interpolation formula}.

\begin{theorem}
    Let $\omega_{\mathbf{g}}, \omega_{\mathbf{f}}$ be Coleman families deforming (fixed) small slope $p$-stabilisations of $g$ and $f$, defined over connected rigid spectrum $\Sp{(\Lambda_g)}$ and $\Sp{(\Lambda_f)}$, such that their weights satisfy some mild analyticity hypothesis (cf. Assumption \ref{assumption: weight r}). There is a 3-variable rigid meromorphic function $\mathscr{L}_p(\omega_{\mathbf{g}}, \omega_{\mathbf{f}})$ defined on $\Sp{(\Lambda_g \hat{\otimes} \Lambda_f)}$, such that for all pairs of $F$-dominated weights $(x, c)$ and $(y, 2c)$ with $y = (x+t)_{|F}$ that are sufficiently large compared to the slope of $\omega_{\mathbf{f}}$, we have
    \[
    \mathscr{L}_p({\omega}_\mathbf{g}, {\omega}_\mathbf{f})\left((x,c), (y,2c)\right) = \frac{1}{\mathscr{E}({f}^*_y)}\left(\prod_{\mathfrak{p} \emph{\text{ inert}}} \mathscr{E}_{\mathfrak{p}}({g}_x, {f}^*_y) \prod_{\mathfrak{p} \emph{\text{ split}}} \frac{\mathscr{E}_{\mathfrak{p}}({g}_x, {f}^*_y)}{\mathscr{E}_{0,\mathfrak{p}}({g}_x,{f}^*_y)} \right) \times \frac{\langle\zeta^*(\delta^t {g}_x), {f}^*_y\rangle}{\langle {f}^*_y,{f}^*_y\rangle}.
    \]
\end{theorem}

On the other hand, one may also consider the values of $\mathscr{L}_p(\omega_\mathbf{g}, \omega_\mathbf{f})$ outside the $F$-dominated region.
To study this case, we will restrict ourself to $F = \Q$.
We first fix a balanced tuple $((v,n),(w,m))$ of classical weights, i.e. there exists $s \in \N$ such that $(v- (s+1, 0) , n)_{|\Q} = (w,m)$. Let $\ell := 2v + nt_L$ where $t_L$ is the parallel weight, and $k = 2w+m$.
We will further assume that the specialization of the family $\omega_{\mathbf{g}}$ (resp. $\omega_{\mathbf{f}}$) at the weight $\ell$ (resp. $k$) is a $p$-stabilization of a classical modular form, denoted by $g_\ell$ (resp. $f_k$).

Let $K / \mathbb{Q}_p$ be a large enough extension and $\mathcal{O}_K$ be its ring of integers.
Inside the Kuga-Sato variety $\mathscr{U}_{\ell-4} \times_{\mathcal{O}_K} \mathscr{W}_{k-2}$  (cf. \S \ref{subsection: AJ images}), there is a de Rham null-homologous cycle called the generalized \textit{Hirzebruch--Zagier} cycle,
\[ \Delta_{\ell, k} \in \CH^{d-\gamma -1} (\mathscr{U}_{\ell-4} \times \mathscr{W}_{k-2})_0,
\]
where $d$ is the relative dimension of $\mathscr{U}_{\ell-4} \times \mathscr{W}_{k-2}$ and $\gamma = \frac{|\ell| +k -6}{2}$.
The syntomic Abel--Jacobi map $\AJ_p$ then sends $\Delta_{\ell, k} $
to $[ \Fil^{\gamma+2} H^{|\ell|+k -3}_{\dR} (U_{\ell-4} \times_K W_{k-2} / K) ]^\vee$.
A specific element 
$$ \pi_1^* \omega_{g_\ell} \cup \pi_2^* \eta_{f_k} \in \Fil^{\gamma+2} H^{|\ell|+k -3}_{\dR} (U_{\ell-4} \times_K W_{k-2} / K)$$
corresponding to the forms $g_\ell$ and $f_k$ will be given in \S \ref{subsection: AJ images} such that the evaluation 
$$\AJ_p (\Delta_{\ell, k}) (\pi_1^* \omega_{g_\ell} \cup \pi_2^* \eta_{f_k})$$ 
makes sense.

Our main results are the following two theorems, where the various Euler factors will be introduced in \S \ref{section: GZ-formula}.

\begin{theorem} \label{theorem: main theorem split case}
    Suppose $p$ is split in $L$.
    Let $(P, Q)$ be a classical point corresponding to a balanced tuple $(\ell, k)$ as above. 
    Then we have 
\begin{equation*}
    \mathscr{L}_p(\omega_\mathbf{g}, \omega_\mathbf{f})(P, Q) = \frac{(-1)^s}{s! \mathscr{E}(f_k^*)} \frac{\mathscr{E}_p (g_\ell, f_k^*)}{\mathscr{E}_{0, p} (g_\ell, f_k^*) } \AJ_p (\Delta_{\ell, k}) (\pi_1^* \omega_{g_\ell} \cup \pi_2^* \eta_{f_k}).
\end{equation*}
\end{theorem}

\begin{theorem} \label{theorem: main theorm inert case}
    Suppose $p$ is inert in $L$.
    Let $(P, Q)$ be a classical point corresponding to a balanced tuple $(\ell, k)$ as above. 
    Then we have 
    \begin{equation*}
    \mathscr{L}_p(\omega_{\mathbf{g}}, \omega_{\mathbf{f}})(P, Q) = \frac{(-1)^s}{s! \mathscr{E}(f_k^*)} {\mathscr{E}_p(g_\ell, f_k^*)} \AJ_p (\Delta_{\ell, k}) (\pi_1^* \omega_{g_\ell} \cup \pi_2^* \eta_{f_k}).
    \end{equation*}
\end{theorem}

We remark that we are able to prove the formula in the inert case.
The reason is that if one works with the sheaf of nearly overconvergent forms in \cite{twisted_tripleL_finite_slope}, one is able to $p$-adically iterate the Gauss--Manin connection $\nabla$.
It is then easy to obtain a primitive of a $p$-depleted form $\omega_{\mathbf{g}}^{[\mathcal{P}]}$ by simply taking $\nabla^{-1} \omega_{\mathbf{g}}^{[\mathcal{P}]}$ (compare this to \cite[Corollary~4.7]{blanco-chacón_fornea_2020}).

With our $p$-adic Gross--Zagier formulae, one is then able to generalise the result of \cite[Corollary~6.1]{blanco-chacón_fornea_2020} on the dimension of Selmer group, to the case where $f$ is of finite slope at $p$ and $p$ is unramified in $L$.

Another application is the comparison between the automorphic and motivic $L$-functions analogous to \cite[Theorem~10.1]{Fornea_Jin_2024}, when the families are of finite slope at $p$.
In other words, one can interpret the $p$-adic $L$-function $\mathscr{L}_p(\omega_{\mathbf{g}}, \omega_{\mathbf{f}})$ as an image of a certain class under the big logarithm of Perrin-Riou.
The proof is essentially the same as in \textit{loc. cit.},
so we decided to omit it.


\vspace{2pt}

\paragraph{\textbf{Notations}.} 
Throughout this article, we will frequently consider families of modular forms and classical modular forms.
We will primarily use the notations ${\mathbf{g}}, {\mathbf{f}}$ and $\omega_{\mathbf{g}}, \omega_{\mathbf{f}}$ for families of modular forms, while we use $g, f$ and $\omega_g, \omega_f$ for classical modular forms.
The notation $\langle \phantom{e}, \phantom{e} \rangle$ will be reserved for the Petersson product, which is linear in the first factor (contrary to the convention in \cite{DR}).
When we deal with the Poincar\'{e} pairing, we will always use $\langle \phantom{e}, \phantom{e} \rangle_{\dR}$ to distinguish it from the Petersson product.

\vspace{2pt}

\paragraph{\textbf{Acknowledgements}.} 
This project started when both the authors were at Concordia University, where the first author was a PhD student and the second author was a postdoctoral researcher. 
We would like to thank Michele Fornea, David Loeffler, and Giovanni Rosso for their suggestions on an earlier version of the paper. We would also like to thank Fabrizio Andreatta and Adrian Iovita for encouraging us to work on this problem.
The first author was supported by NSERC grants RGPIN 2018-04392, ALLRP 577144-22 and FRQNT grant 2019-NC-254031 when he was at Concordia, and was supported by the grant G5392391GRA3C565LAGAXR from the ANR-DFG Project HEGAL as a postdoctoral researcher in LAGA, Universit\'{e} Paris XIII. The second author gratefully acknowledges the support of an NSERC grant during his time as a postdoc at Concordia, as well as of the European Research Council's Consolidator Grant “ShimBSD: Shimura varieties and the BSD conjecture”, grant ID 101001051, that he received as a postdoc at UniDistance Suisse through the Horizon 2020 Excellent Science programme.

\newcommand{\Sh}[1]{\mathcal{#1}}
\newcommand{\Gm}{\mathbb{G}_m}
\newcommand{\Ga}{\mathbb{G}_a}
\newcommand{\Xrig}[2]{\bar{\Sh{X}}^{G_{#1}}_{#2}}
\newcommand{\Xfor}[2]{\bar{\mathfrak{X}}^{G_{#1}}_{#2}}
\newcommand{\Hdg}[1]{\mathrm{Hdg}^{#1}}
\newcommand{\V}[3]{\mathbb{V}_0(#1,#2,#3)}
\newcommand{\Hs}[1]{\Sh{H}^{\sharp}_{\Sh{A}#1}}

\section{Twisted triple product \texorpdfstring{$p$}{p}-adic  \texorpdfstring{$L$}{L}-function}\label{section: L-function}

In this section we briefly recall the definition of the twisted triple product $L$-function, both the classical one and its $p$-adic avatars. 

Let $L/F$ be a quadratic extension of totally real number fields, with $[F:\Q] = d$. Assume $p$ is a rational prime that is unramified in $L$. Let $\Sigma_L$ and $\Sigma_F$ be their respective sets of embeddings into a common finite extension $K$ of $\Q_p$, where both $L$ and $F$ are split. Let $G_L = \Res_{L/\Q}\mathrm{GL}_{2,L}$, and $G_F = \Res_{F/\Q}\mathrm{GL}_{2,F}$.

\subsection{The weight space}
Here we recall the definition of the weight space for Hilbert modular forms, for the group $G_F$. Everything generalizes analogously for $G_L$. 

Let $\mathbb{T} = \Res_{\Sh{O}_F/\Z}\Gm$, so that $\mathbb{T}(\Z_p) = (\Sh{O}_F \otimes \Z_p)^{\times}$. 
Let $\Lambda^{G_F} := \Sh{O}_K\llbracket \mathbb{T}(\Z_p) \times \Z_p^{\times}\rrbracket$, $\mathfrak{W}^{G_F} = \Spf{\Lambda^{G_F}}$, and $\Sh{W}^{G_F} = \Spa(\Lambda^{G_F},\Lambda^{G_F})^{\text{an}}$ be the associated analytic adic space.
Let $(v^{\text{un}}, w^{\text{un}}) \in \Hom_{\text{cont}}((\Sh{O}_F\otimes \Z_p)^{\times} \times \Z_p^{\times}, (\Lambda^{G_F})^{\times})$ be the universal character.

There is another auxiliary weight space $\Sh{W}_F := \Spa(\Sh{O}_K\llbracket (\Sh{O}_F \otimes \Z_p)^{\times} \rrbracket)^{\text{an}}$ and a map $\Sh{W}^{G_F} \to \Sh{W}_F$ induced by the algebra map that is described on the group-like elements $(\Sh{O}_F \otimes \Z_p)^{\times} \to (\Sh{O}_F \otimes \Z_p)^{\times} \times \Z_p^{\times}$ as sending $t \mapsto (t^2, \mathrm{Nm}_{F/Q}(t))$. Equivalently, on the $\C_p$ points of the weight spaces the map can be described as sending a weight $(w, m) \in \Sh{W}^{G_F}(\C_p) \mapsto 2w + mt_F \in \Sh{W}_F(\C_p)$ where $t_F = (1, \dots, 1) \in \Z^{\Sigma_F}$ is the parallel weight $1$. 

For a $p$-adically complete flat $\Sh{O}_K\llbracket (\Sh{O}_F \otimes \Z_p)^{\times}\rrbracket$-algebra $R$ and a weight $k\colon (\Sh{O}_F \otimes \Z_p)^{\times} \to R^{\times}$, we call $k$ an $m$-analytic character if $k = \chi\cdot k^0$ where $\chi = k_{|(\Sh{O}_F \otimes \Z_p)^{\times}_{\mathrm{tor}}}$ is the finite part of the character and $k^0$ is analytic on $1+p^m(\Sh{O}_F \otimes \Z_p)$. We note that there are open subspaces $\Sh{W}_{F,m}$ of the weight space $\Sh{W}_F$ where the universal character satisfies this ``$m$-analyticity" property.

\subsection{Generalities on Hilbert modular forms} \label{subsection: Hilbert modular forms}
There are two related notions of Hilbert modular forms: \emph{arithmetic} Hilbert modular forms which are directly related to automorphic representations, and hence suited for arithmetic applications like the study of $L$-functions; \emph{geometric} Hilbert modular forms which have a modular interpretation as sections of modular sheaves on Hilbert modular varieties, and hence more suitable for geometric constructions. But before we explain these notions, let us provide a simpler definition of the arithmetic Hilbert modular forms using the language of automorphic forms as functions on the adelic points of $G_L$ (or $G_F$).

Let $K_1(N_g) = \left\{\left(\begin{smallmatrix}
    a & b \\
    c & d
\end{smallmatrix}\right) \in \mathrm{GL}_2(\hat{\Sh{O}}_L) \, | \, d \equiv 1, c \equiv 0 \text{ mod } N_g\right\}$ for an ideal $N_g \subset \Sh{O}_L$. 
A Hilbert modular cuspform $g$ of weight $(v,n) \in \Z[\Sigma_L] \times \Z$  is a function $g \colon G_L(\Q)\backslash G_L(\A{}{})/K_1(N_g) \to \C{}$ that satisfies the following properties:
\begin{enumerate}
    \item For every finite adelic point $x \in G_L(\A{}{\infty})$ (the superscript denotes ``away from the $\infty$-places"), the well-defined function $f_x \colon \mathfrak{h}^{2d} \to \C$ given by $f_x(z) = f(u_{\infty}x)j_{k,v}(u_{\infty}, \mathbf{i})$ is holomorphic. Here $\mathfrak{h}$ is the usual Poincar\'{e} upper half plane, $u_{\infty} = \left(\begin{smallmatrix}
        a & b \\
        c & d
    \end{smallmatrix}\right) \in \GL{2}(L\otimes \R)_{+}$ such that $u_{\infty} \cdot \mathbf{i} = z$ under the usual M\"{o}bius action, $k = 2v + nt_L$, and $j_{k,v}(u_{\infty}, z) = (ad-bc)^{-v}(cz + d)^k$ is the automorphic factor.
    \item $f(xu_{\infty}) = f(x)j_{k,v}(u_{\infty}, \mathbf{i})^{-1}$ for any $u_{\infty} \in \mathrm{Stab}(\mathbf{i}) \cap G_L(\R)_+$.
    \item For all adelic points $x \in \GL{2}(\A{L}{})$, and for all additive measures $da$ on $L \backslash \A{L}{}$, we have 
    \[
     \int_{L\backslash\A{L}{}} \phi_f\left(\left(\begin{smallmatrix}
        1 & a \\
        0 & 1
    \end{smallmatrix}\right)x\right)da = 0.
    \]
\end{enumerate}

Let us now quickly set up the geometric setting that will be useful later to talk about $p$-adic interpolation of modular forms. Let $M^{G_F}_{\Q_p}$ be the (base change to $\Q_p$ of the) Shimura variety associated with $G_F$ of level $K$ for a neat open compact $K = K^{(p)}K_p \subset G_F(\hat{\Z})$ such that $K_p = G_F(\Z_p)$. The Shimura variety $M^{G_F}_{\Q_p}$ is not connected and its connected components are parametrised by the strict class group $\mathrm{cl}^+_F(K) := F^{\times,+}\backslash \A{F,f}{\times}/\det K$. The connected components of $M^{G_F}_{\Q_p}$ are finite \'{e}tale quotients of certain Hilbert modular varieties which we explain below.

For a fractional ideal $\mathfrak{c}$ of $F$, let $X^{G_F, \mathfrak{c}}/\Spec{\Z_p}$ be the Hilbert modular scheme classifying tuples $(A, \iota, \lambda, \alpha_K)$ where $A$ is an abelian variety of dimension $[F:\Q]$ with real multiplication structure given by $\iota \colon \Sh{O}_F \to \mathrm{End}(A)$, $\lambda \colon (\Hom^{\text{sym}}_{\Sh{O}_F}(A, A^{\vee}), \mathrm{Pol}_{\Sh{O}_F}(A)) \xrightarrow{\sim} (\mathfrak{c}, \mathfrak{c}^{+})$ is a $\mathfrak{c}$-polarisation structure in the sense of \cite{DelignePappas}, and $\alpha_K$ is a $K$-level structure. There is a natural action of the totally positive units $\Sh{O}_F^{\times, +}$ on $X^{G_F,\mathfrak{c}}/\Spec{\Z_p}$ given by $\epsilon \cdot (A,\iota, \lambda, \alpha_K) = (A, \iota, \epsilon\lambda, \alpha_K)$. This action factors through the finite quotient $\Delta_K := \Sh{O}_F^{\times, +}/(K \cap \Sh{O}_F^{\times})^2$. The geometrically connected components of $X^{G_F,\mathfrak{c}}$ are parametrised by $\hat{\Sh{O}}_F^{\times}/\det K$, and the stabiliser of each geometrically connected component for the action of $\Delta_K$ is $(\det K \cap \Sh{O}_F^{\times,+})/(K\cap \Sh{O}_F^{\times})^2$. Therefore, the geometrically connected components of $X^{G_F,\mathfrak{c}}/\Delta_K$ are parametrised by $\Sh{O}_F^{\times,+}\backslash \hat{\Sh{O}}_F^{\times}/\det K = \ker(\mathrm{cl}^+_F(K) \to \mathrm{cl}^+_F)$ where $\mathrm{cl}^+_F = \mathrm{cl}^+_F(G_F(\hat{\Z}))$ is the strict class group. 

Let $\mathfrak{N}$ be the modulus of $\mathrm{cl}^+_F(K)$. Consider the group $I_{\mathfrak{N}p}$ of fractional ideals of $F$ coprime to $\mathfrak{N}p$, and let $P_{\mathfrak{N}p} = \ker(I_{\mathfrak{N}p} \to \mathrm{cl}^+_F)$ be the subgroup of ideals generated by $x \in F^{\times,+}$ such that $x$ is coprime to $\mathfrak{N}p$. Let $X^{G_F}/\Spec \Z_p = \sqcup_{\mathfrak{c} \in I_{\mathfrak{N}p}} X^{G_F,\mathfrak{c}}$. Let $F^{\times,+}_{\mathfrak{N}p}$ be the group of totally positive $x \in F^{\times,+}$ such that $x$ is coprime to $\mathfrak{N}p$.  There is a natural action of $F^{\times, +}_{\mathfrak{N}p}$ on $X^{G_F}$ given by $x \cdot (A, \iota, \lambda, \alpha_K) = (A, \iota, x\lambda, \alpha_K)$. For an ideal $(x) \in P_{\mathfrak{N}p}$, multiplying the polarisation isomorphism by $x$ induces an isomorphism $X^{G_F,\mathfrak{c}}/\Delta_K \to X^{G_F, x\mathfrak{c}}/\Delta_K$. Thus $M^{G_F} := X^{G_F}/F^{\times,+}_{\mathfrak{N}p}$ makes sense as a finite type scheme over $\Spec{\Z_p}$, whose geometrically connected components are parametrised by $\mathrm{cl}^+_F(K)$. One has that $M^{G_F}$ is an integral model of the Shimura variety $M^{G_F}_{\Q_p}$ \cite[\S2.3]{TianXiao}.

\begin{remark} \label{remark: Hilbert modular scheme, G*}
    For $\mathfrak{d}_F$ the different ideal of $F$, the Hilbert modular scheme $X^{G_F, \mathfrak{d}_F^{-1}}$ is an integral model of the Shimura variety for the group $G_F^* := G_F \times_{\Res_{F/\Q}\Gm} \Gm$, where the map $G_F \to \Res_{F/\Q}\Gm$ is the determinant map, and $\Gm \to \Res_{F/\Q}\Gm$ is the diagonal embedding.
\end{remark}

Choose a smooth toroidal compactification $\bar{X}^{G_F,\mathfrak{c}}$ of $X^{G_F,\mathfrak{c}}$ for each $\mathfrak{c}$, compatibly so that the multiplication by $x$ action on the polarisation module for $x \in F^{\times, +}_{\mathfrak{N}p}$ induces an isomorphism $\bar{X}^{G_F,\mathfrak{c}}\simeq \bar{X}^{G_F,x\mathfrak{c}}$. The action of $(\det K \cap \Sh{O}_F^{\times, +})/(K\cap \Sh{O}_F^{\times})^2$ on the geometrically connected components of ${X}^{G_F,\mathfrak{c}}$ extends to a free action on the geometrically connected components of $\bar{X}^{G_F,\mathfrak{c}}$ \cite[\S8.1]{andreatta2016adic}, and the finite \'{e}tale quotient $\bar{M}^{G_F} := \bar{X}^{G_F}/F^{\times,+}_{\mathfrak{N}p}$ defines a smooth toroidal compactification of $M^{G_F}$.
We similarly define $X^{G_L}, M^{G_L}$ and their compactified counterparts.

Let $\Sh{A} \to \bar{X}^{G_F}$ be the semi-abelian scheme extending the universal abelian scheme over $X^{G_F}$. Let $e \colon \bar{X}^{G_F} \to \Sh{A}$ be the unit section. Let $\omega_{\Sh{A}} := e^*\Omega^1_{\Sh{A}/\bar{X}^{G_F}}$ be the sheaf of invariant differentials. Let $\Sh{H}_{\Sh{A}}$ be the canonical extension of the relative de Rham sheaf $\Sh{H}^1_{\dR}(\Sh{A}/X^{G_F})$ to $\bar{X}^{G_F}$. We have an exact sequence known as the Hodge filtration:
\[
0 \to \omega_{\Sh{A}} \to \Sh{H}_{\Sh{A}} \to \omega_{\Sh{A}^{\vee}}^{\vee} \to 0.
\]

Recall that $K$ was chosen to be a finite extension of $\Q_p$ that splits $F$ and $L$. Fix any $\Sh{O}_K$-algebra $R$, and base change $\bar{X}^{G_F}$ to $\Spec R$. The modules $\omega_{\Sh{A}}$ and $\Sh{H}_{\Sh{A}}$ are $\Sh{O}_F \otimes \Sh{O}_{\bar{X}^{G_F}}$-modules locally free of rank 1 and 2 respectively. Hence using the distinct embeddings $\sigma \colon \Sh{O}_F \to R \to \Sh{O}_{\bar{X}^{G_F}}$, we get a decomposition of the Hodge filtration as a direct sum over $\Sigma_F$:
\[
0 \to \omega_{\Sh{A},\sigma} \to \Sh{H}_{\Sh{A},\sigma} \to \omega_{\Sh{A}^{\vee},\sigma}^{\vee} \to 0,
\]
where $\omega_{\Sh{A},\sigma}$ and $\Sh{H}_{\Sh{A},\sigma}$ are locally free $\Sh{O}_{\bar{X}^{G_F}}$-modules of rank 1 and 2 respectively.

\begin{definition}
    For a weight ${k} = (k_{\sigma})_{\sigma \in \Sigma_F} \in \Z^{\Sigma_F}$, the sheaf of geometric Hilbert modular forms of weight $k$ is defined as 
    \[
    \omega_{\Sh{A}}^{{k}} := \bigotimes_{\sigma \in \Sigma_F} \omega_{\Sh{A},\sigma}^{k_{\sigma}}.
    \]
    The symmetric ${k}$-th power of $\Sh{H}_{\Sh{A}}$ is defined similarly as
    \[
    \Sym^{{k}}\Sh{H}_{\Sh{A}} := \bigotimes_{\sigma \in \Sigma_F} \Sym^{k_{\sigma}}\Sh{H}_{\Sh{A},\sigma}.
    \]
\end{definition}

Let now $({w},m) \in \Z^{\Sigma_F} \times \Z$ be a weight for the group $G_F$. The sheaf of arithmetic Hilbert modular forms of weight $({w},m)$, denoted by $\omega_{\Sh{A}}^{(w,m)} $, is a sheaf on $\bar{M}^{G_F}$ that is defined by descending the sheaf $\omega_{\Sh{A}}^{2w+mt_F}$ on $\bar{X}^{G_F}$ along the \'{e}tale projection $\bar{X}^{G_F} \to \bar{M}^{G_F}$ as we explain below.

Consider the $\Sh{O}_F \otimes \Sh{O}_{\bar{X}^{G_F}}$-line bundle $\wedge^2_{\Sh{O}_F} \Sh{H}_{\Sh{A}}$. The polarisation on $\bar{X}^{G_F,\mathfrak{c}}$ induces an isomorphism $\omega_{\Sh{A}^{\vee}}^{\vee} \xrightarrow{\sim} \omega_{\Sh{A}}^{\vee} \otimes \mathfrak{c}$ \cite[(1.0.15)]{KatzNicholasM1978pLfC}, which therefore gives a canonical isomorphism $\wedge^2_{\Sh{O}_F}\Sh{H}_{\Sh{A}} = \omega_{\Sh{A}} \otimes_{\Sh{O}_F} \omega^{\vee}_{\Sh{A}^{\vee}} \simeq \omega_{\Sh{A}} \otimes_{\Sh{O}_F} (\omega_{\Sh{A}}^{\vee} \otimes \mathfrak{c}) = \Sh{O}_{\bar{X}^{G_F,\mathfrak{c}}}\otimes \mathfrak{cd}^{-1}$.

For $x \in F^{\times,+}_{\mathfrak{N}p}$, the isomorphism $[x] \colon \bar{X}^{G_F,\mathfrak{c}} \xrightarrow{\sim} \bar{X}^{G_F,x\mathfrak{c}}$ defined by $(A,\iota, \lambda, \alpha_K) \mapsto (A,\iota,x\lambda,\alpha_K)$ induces a canonical isomorphism $[x]^*\wedge^2_{\Sh{O}_F}\Sh{H}_{\Sh{A}} \to \wedge^2_{\Sh{O}_F}\Sh{H}_{\Sh{A}}$ such that the following diagram commutes.
\[\begin{tikzcd}
	{[x]^*(\wedge^2_{\Sh{O}_F}\Sh{H}_{\Sh{A}})} & {\wedge^2_{\Sh{O}_F}\Sh{H}_{\Sh{A}}} \\
	{\Sh{O}_{M^{\mathfrak{c}}}\otimes x\mathfrak{cd}^{-1}} & {\Sh{O}_{M^{\mathfrak{c}}}\otimes \mathfrak{cd}^{-1}}
	\arrow[from=1-1, to=1-2]
	\arrow["\simeq", from=1-1, to=2-1]
	\arrow["\simeq", from=1-2, to=2-2]
	\arrow["{x^{-1}}", from=2-1, to=2-2]
\end{tikzcd}\]
Letting $k = 2w+mt_F$.
This defines a descent datum on the sheaf $\omega_{\Sh{A}}^{k} \otimes (\wedge^2_{\Sh{O}_F}\Sh{H}_{\Sh{A}})^{-w} := \omega_{\Sh{A}}^{k}\otimes \left(\otimes_{\sigma \in \Sigma_F}(\wedge^2 \Sh{H}_{\Sh{A},\sigma})^{-w_{\sigma}}\right)$, and the sheaf $\omega_{\Sh{A}}^{(w,m)}$ of arithmetic Hilbert modular forms of weight $(w,m)$ is defined to be the descent of $\omega_{\Sh{A}}^{k}\otimes (\wedge^2_{\Sh{O}_F} \Sh{H}_{\Sh{A}})^{-w}$ along the \'etale quotient $\bar{X}^{G_F} \to \bar{M}^{G_F}$ with respect to the datum above. We however choose a different formulation of the descent datum that will be easier to extend to the setting of interpolation sheaves later. Note that the canonical trivialisation $\Sh{O}_{\bar{X}^{G_F,\mathfrak{c}}} \otimes \mathfrak{cd}^{-1} \simeq \wedge^2_{\Sh{O}_F} \Sh{H}_{\Sh{A}}$ defines a canonical section $\eta_{\lambda, \mathfrak{c}} \in \wedge^2_{\Sh{O}_F}\Sh{H}_{\Sh{A}}$ as the image of $1$. This gives an isomorphism 
\begin{align*}
    \omega_{\Sh{A}}^{k} \otimes (\wedge^2_{\Sh{O}_F}\Sh{H}_{\Sh{A}})^{-w} &\xrightarrow{\sim} \omega_{\Sh{A}}^{k} \\
    f &\mapsto \left[\tilde{f} \colon (A,\iota,\lambda,\alpha_K,\omega) \mapsto f(A,\iota, \lambda, \alpha_K, \omega, \eta_{\lambda,\mathfrak{c}})\right].
\end{align*}
Note that $\eta_{x\lambda, x\mathfrak{c}} = x^{-1}\eta_{\lambda, \mathfrak{c}}$. Therefore the descent datum above can be translated into a descent datum on the sheaf $\omega_{\Sh{A}}^{k}$, which can be described as the following action of $F^{\times, +}_{\mathfrak{N}p}$.
\begin{align*}
    F^{\times,+}_{\mathfrak{N}p} \times \omega_{\Sh{A}}^{k} &\xrightarrow{} \omega_{\Sh{A}}^{k} \\
    (x, f) &\mapsto \left[ (x \ast f) \colon (A, \iota, \lambda, \alpha_K, \omega) \mapsto x^{w}f(A, \iota, x\lambda, \alpha_K, \omega)\right].
\end{align*}
It is easy to see that $(K\cap \Sh{O}_F^{\times})^2$ acts trivially. Indeed, for an element $\epsilon \in K \cap \Sh{O}_F^{\times}$, pullback along the isomorphism $[\epsilon^{-1}] \colon \Sh{A} \to \Sh{A}$ induces an isomorphism of tuples $(A, \iota, \lambda, \alpha_K, \epsilon \omega) \simeq (A, \iota, \epsilon^2\lambda, \alpha_K, \omega)$. Hence $(\epsilon^2\cdot f)(A, \iota, \lambda, \alpha_K, \omega) = \epsilon^{2w}f(A, \iota, \epsilon^2\lambda, \alpha_K, \omega) = \epsilon^{2w}f(A, \iota, \lambda, \alpha_K, \epsilon\omega) = \epsilon^{-mt_F}f(A, \iota, \lambda, \alpha_K, \omega) = f(A, \iota, \lambda, \alpha_K, \omega)$. Here we assume that either $\mathrm{N}_{F/\Q}(K \cap \Sh{O}_F^{\times}) = \{1\}$ or $m \in 2\Z$ because otherwise the notion of arithmetic Hilbert modular forms is empty.
\begin{definition}
    Let $p \colon \bar{X}^{G_F} \to \bar{M}^{G_F}$ be the projection. We define the sheaf of arithmetic Hilbert modular forms of weight $(w,m)$ to be $\omega_{\Sh{A}}^{(w,m)} := (p_*\omega_{\Sh{A}}^{k})^{F^{\times, +}_{\mathfrak{N}p}}$, i.e. the sheaf of invariant sections for the action of $F^{\times, +}_{\mathfrak{N}p}$ defined above.
\end{definition}

\subsection{Twisted triple product \texorpdfstring{$L$}{L}-function}
Here we recall the classical $L$-function following \cite[\S3]{blanco-chacón_fornea_2020}.

Let $(v,n) \in \Z[\Sigma_L] \times \Z$ and $(w,m) \in \Z[\Sigma_F] \times \Z$ be two classical weights. Consider primitive eigenforms $g \in S^{G_L}(N_g, (v,n), \bar{\Q})$ and $f \in S^{G_F}(N_f, (w,m), \bar{\Q})$ generating irreducible cuspidal automorphic representations $\pi, \sigma$ of $G_L(\A{}{})$ and $G_F(\A{}{})$ respectively. Here $S^{G_L}(N_g, (v,n), \bar{\Q})$ denotes the space of \emph{arithmetic} Hilbert cuspforms for the group $G_L$ of $\Gamma_1(N_g)$-level and weight $(v,n)$ with coefficients in $\bar{\Q}$. The $\bar{\Q}$-module $S^{G_F}(N_f, (w,m), \bar{\Q})$ is to be understood similarly. Let $\pi^u = \pi \otimes | \cdot |_{\A{L}{}}^{n/2}$ and $\sigma^u = \sigma \otimes | \cdot |_{\A{F}{}}^{m/2}$ be the unitarisation of $\pi$ and $\sigma$. Consider the representation $\Pi = \pi^u \otimes \sigma^u$. Let $\rho \colon \Gal(\bar{F}/F) \to S_3$ be the homomorphism mapping the absolute Galois group of $F$ to the symmetric group over 3 elements associated to the \'{e}tale cubic algebra $L \times F/F$. Let $G_{L\times F} = \Res_{L/F}\GL{2,L} \times \GL{2,F}$. The $L$-group ${}^{L}(G_{L\times F})$ is given by the semi-direct product of $\hat{G} \rtimes \Gal(\bar{F}/F)$, where $\Gal(\bar{F}/F)$ acts on $\hat{G} = \GL{2}(\C)^{3}$ through $\rho$. We will assume the central character of $\Pi$ is trivial when restricted to $\A{F}{\times}$.

\begin{definition}
    The twisted triple product $L$-function associated with the unitary automorphic representation $\Pi$ is given by the Euler product
    \[
    L(s,\Pi, \mathrm{R}) = \prod_{v} L_v(s, \Pi_v, \mathrm{R})^{-1},
    \]
    where $\Pi_v$ is the local representation at the place $v$ of $F$ appearing in the restricted tensor product decomposition $\Pi = \otimes'_v \Pi_v$, and $\mathrm{R}$ is the representation of the $L$-group acting on $\C^2 \otimes \C^2 \otimes \C^2$ as follows: $\hat{G}$ acts via the natural action, and $\Gal(\bar{F}/F)$ acts by permuting the vectors through $\rho$.
\end{definition}

\begin{definition}
    Let $(v,n) \in \Z[\Sigma_L] \times \Z$ and $(w,m) \in \Z[\Sigma_F] \times \Z$ be two classical weights. We say the pair $(v,n)$ and $(w,m)$ is $F$-dominated if there exists $r \in \N[\Sigma_L]$ such that $w = (v+r)_{|F}$, and $mt_F = (nt_L)_{|F}$ (or equivalently, $m = 2n$).
\end{definition}

For $\phi \in \Pi$, consider the period integral $I(\phi) := \int_{\A{F}{\times}\GL{2}(F)\backslash\GL{2}(\A{F}{})}\phi(x)\mathrm{d}x$.

\begin{prop}\label{P203}
    Let $\eta \colon \A{F}{\times} \to \C^{\times}$ be the quadratic character attached to $L/F$ by class field theory. Then the following are equivalent:
    \begin{enumerate}
        \item The central $L$-value $L(\frac{1}{2}, \Pi, \mathrm{R})$ does not vanish, and for every place $v$ of $F$ the local $\epsilon$-factor satisfies $\epsilon_v(\frac{1}{2}, \Pi_v, \mathrm{r})\eta_v(-1) = 1$.
        \item There exists a vector $\phi \in \Pi$, called a test vector, whose period integral $I(\phi)$ does not vanish.
    \end{enumerate}
    Moreover, in this situation, 
    \[
    L\left(\frac{1}{2}, \Pi, \mathrm{R}\right) = (\ast)\cdot |I(\phi)|^2
    \]
    where $(\ast)$ is some non-zero constant.
\end{prop}

\begin{proof}
    \cite[Theorem 3.2]{blanco-chacón_fornea_2020}.
\end{proof}

\begin{remark}
    In \cite[1963]{blanco-chacón_fornea_2020} the authors provide sufficient criteria on $g$ and $f$ for the equivalent conditions of Proposition \ref{P203} to hold. If the weights $(v,n)$ and $(w,m)$ are $F$-dominated, then the local $\epsilon$-factors at archimedean places satisfy the hypothesis of the proposition. The same is true for the finite places if we assume that $\mathrm{N}_{L/F}(N_g)\cdot d_{L/F}$ and $N_f$ are coprime, and that every prime dividing $N_f$ splits in $L$. Here $d_{L/F}$ is the discriminant.
\end{remark}

\begin{assumption}
    Assume that the equivalent conditions in Proposition \ref{P203} hold, and the weights $(v,n)$ and $(w,m)$ of $g$ and $f$ respectively are $F$-dominated.
\end{assumption}

Let $\mathfrak{A}$ be an ideal in $\Sh{O}_F$. Denote by $K_{11}(\mathfrak{A})$ the following open compact group.
\[
K_{11}(\mathfrak{A}) = \left\{ \begin{pmatrix}
    a & b \\
    c & d
\end{pmatrix} \in \GL{2}(\hat{\Sh{O}}_F)\,|\, a\equiv d\equiv 1, c\equiv 0 \text{ mod } \mathfrak{A} \right\}.
\]

Let $\mathfrak{J}$ be the element
\[
{\begin{pmatrix}
-1 & 0 \\
0 & 1
\end{pmatrix}}^{\Sigma_F} \in \GL{2}(\R)^{\Sigma_F}.
\]
For any $h \in \sigma^u$, we define $h^{\mathfrak{J}} \in \sigma^u$ to be the vector obtained by right translation $h^{\mathfrak{J}}(x) = h(x\mathfrak{J})$. 

Let $\delta$ be the Maass--Shimura differential operator \cite[1956]{blanco-chacón_fornea_2020} on nearly holomorphic cuspforms. 

The natural inclusion $\GL{2}(\A{F}{}) \to \GL{2}(\A{L}{})$ defines by composition a \emph{diagonal restriction} map \[\zeta^* \colon S^{G_L}(K_{11}(\mathfrak{A}\Sh{O}_L), (v,n), \C{}) \to S^{G_F}(K_{11}(\mathfrak{A}), (v_{|F}, 2n), \C{}).\]

\begin{lemma}\label{L206}
Let $r \in \N[\Sigma_L]$ be such that $w = (v+r)_{|F}$. Then there is an $\Sh{O}_F$ ideal $\mathfrak{A}$ supported on a subset of the prime factors of $N_f\cdot\mathrm{N}_{L/F}(N_g)\cdot d_{L/F}$ such that a test vector $\phi$ as in the statement of Proposition \ref{P203} can be chosen to be of the form $\phi = (\delta^r(\breve{g}))^u \otimes (\breve{f}^{\mathfrak{J}})^u$ for $\breve{g} \in S^{G_L}(K_{11}(\mathfrak{A}\Sh{O}_L), (v,n), \bar{\Q})$ and $\breve{f} \in S^{G_F}(K_{11}(\mathfrak{A}), (w,m), \bar{\Q})$. The cuspforms $\breve{g}, \breve{f}$ are eigenforms for all Hecke operators outside $N_f\cdot\mathrm{N}_{L/F}(N_g)\cdot d_{L/F}$ with the same eigenvalues of $g$ and $f$ respectively. Moreover, in this case the period integral can be expressed as
\[
I(\phi) = \int_{[\GL{2}(\A{F}{})]}(\delta^r\breve{g})^u \otimes (\breve{f}^{\mathfrak{J}})^u\mathrm{d}x = \langle \zeta^*(\delta^r \breve{g}), \breve{f}^*\rangle
\]
where $\breve{f}^*$ is the cuspform in $S^{G_F}(K_{11}(\mathfrak{A}), (w,m), \bar{\Q})$ whose adelic $q$-expansion coefficients are complex conjugates of those of $\breve{f}$. Here $\langle \phantom{e}, \phantom{e} \rangle$ is the Petersson inner product.
\end{lemma}

\begin{proof}
\cite[Lemma 3.4]{blanco-chacón_fornea_2020}.
\end{proof}

\begin{prop}\label{P2010}
    The value $L^{\text{alg}}(g,f) := \frac{\langle \zeta^*(\delta^r \breve{g}), \breve{f}^*\rangle}{\langle f^*, f^*\rangle}$ is algebraic. We have
    \[
    L\left(\frac{1}{2}, \Pi, \mathrm{R}\right) = (\ast)\cdot|L^{\text{alg}}(g,f)|^2
    \]
    where $(\ast)$ is a constant which can be ensured to be non-zero by suitably choosing the test vectors $\breve{g}, \breve{f}$.
\end{prop}

\begin{proof}
    The first statement is \cite[Proposition 3.5]{blanco-chacón_fornea_2020}. The second statement follows from the proof of Proposition \ref{P203} using Ichino's formula.
\end{proof}

\subsection{Families of (nearly) overconvergent Hilbert modular forms}
Let us now assume that $N_g, N_f$ are coprime to $p$. Let $\bar{\Sh{M}}^{G_L}_{\nu} \to \Sp{\Lambda_g}$ be an overconvergent neighbourhood of the ordinary locus in a smooth toroidal compactification of the rigid analytic Shimura variety $\Sh{M}^{G_L}$ of level $K_{11}(\mathfrak{A})$ defined by the inequality $|\mathrm{Ha}|^{p^{\nu+1}} \geq |p|$, where $\mathrm{Ha}$ denotes (local lifts of) the Hasse invariant, and $\nu \in \N$. One of the main results of \cite{twisted_tripleL_finite_slope} is the following.

\begin{theorem}\label{L2011}
    There exists a Banach sheaf $\mathbb{W}^{G_L}_{k+2r}(-D)$ over $\bar{\Sh{M}}^{G_L}_{\nu}$ interpolating nearly overconvergent cuspforms of weight $k+2r$, which is equipped with $(\mathrm{i})$ the $U$-operator that acts compactly on the space of sections, and $(\mathrm{ii})$ the Gauss--Manin connection $\nabla$ induced by the Gauss--Manin connection on $\Sh{H}_{\Sh{A}};$ such that for any overconvergent family of cuspforms $\breve{\omega}_\mathbf{g}^{[\Sh{P}]}$ of weight $k$, where $( \cdot )^{[\Sh{P}]}$ denotes the $\mathfrak{P}_i$-depletion at all primes $\mathfrak{P}_i | p$ in $\Sh{O}_L$,  $\nabla^r(\breve{\omega}_{\mathbf{g}^{[\Sh{P}]}})$ makes sense as a section of $\mathbb{W}^{G_L}_{k+2r}(-D)$ for sufficiently large $\nu$, assuming $k$ and $r$ are sufficiently analytic families of weights.
\end{theorem}

In the following we will briefly recall the definition of the sheaf $\mathbb{W}^{G_L}_{k}$ (and its cuspidal counterpart $\mathbb{W}^{G_L}_{k}(-D)$) for any locally analytic character $k \colon (\Sh{O}_L \otimes \Z_p)^{\times} \to R^{\times}$ to a $p$-adically complete flat $\Sh{O}_K\llbracket (\Sh{O}_L \otimes \Z_p)^{\times}\rrbracket$-algebra $R$. We will also recall local descriptions of these sheaves, and describe the Gauss--Manin connection in terms of these local descriptions at the cusps, which correspond to Tate objects as explained in \cite[\S5.3]{twisted_tripleL_finite_slope}. We won't recall any details about the $U$-operator -- it will be enough to know that it exists and is compact, for which the reader can refer \cite[Lemma 6.2]{twisted_tripleL_finite_slope}.

Let $\bar{\mathfrak{X}}^{G_L}$ be the $p$-adic completion of $\bar{X}^{G_L}$ and let $\bar{\Sh{X}}^{G_L}$ be its rigid generic fibre. Fix $\nu \in \N$, and let $\Xrig{L}{\nu}$ be the overconvergent neighbourhood of the ordinary locus in $\Xrig{L}{}$ defined by $|\mathrm{Ha}|^{p^{\nu}+1} \geq |p|$. Let $\Xfor{L}{\nu}$ be its formal model given by blowing up $\Xfor{L}{}$ along the ideal locally given by $(\Hdg{p^{\nu}+1},p)$ and taking the maximal open where $(\Hdg{\nu},p)$ is generated by $\Hdg{p^{\nu}+1}$. Here $\Hdg{p^{\nu}+1}$ is a local lift of $\mathrm{Ha}^{p^{\nu}+1}$ evaluated at a local basis of $\omega_{\Sh{A}}$. The theory of canonical subgroups ensures that over $\Xfor{L}{\nu}$, the semiabelian scheme $\Sh{A}$ admits the existence of the canonical subgroup $H^{\can}_n \subset \Sh{A}[p^n]$ of level $n$ for any $n \leq \nu$ \cite[Appendice A]{Andreatta2018leHS}. Let $\Sh{IG}_n = \underline{\Isom}(\Sh{O}_L/p^n, H^{\can, D}_n) \to \Xrig{L}{\nu}$ be the torsor of trivialisations of the dual of $H^{\can}_n$, and let $\mathfrak{IG}_n$ be the normalisation of $\Xfor{L}{\nu}$ in $\Sh{IG}_n$.

\begin{prop}
\begin{enumerate}
    \item There exists an integral model $p\omega_{\Sh{A}} \subset \Omega_{\Sh{A}} \subset \omega_{\Sh{A}}$ of $\omega_{\Sh{A}}$ over $\mathfrak{IG}_n$ that is equipped with a marked generator $s \colon \Sh{O}_{\Xfor{L}{\nu}}/p^n\mathrm{Hdg}^{-\frac{p^n}{p-1}} \xrightarrow{\sim} \Omega_{\Sh{A}}/p^n\mathrm{Hdg}^{-\frac{p^n}{p-1}}$ given by the image of the universal point $P^{\mathrm{univ}} \in H_n^{\can,D}(\mathfrak{IG}_n)$ under the linearisation of the Hodge--Tate map 
    \[
    HT_n \colon H_n^{\can, D} \otimes \Sh{O}_{\mathfrak{IG}_n} \to \omega_{H_n} \simeq \omega_{\Sh{A}}/p^n\mathrm{Hdg}^{-\frac{p^n-1}{p-1}}.
    \]
    \item Let $\xi$ be the $\Sh{O}_L \otimes \Sh{O}_{\mathfrak{IG}_n}$ invertible ideal such that $\Omega_{\Sh{A}} = \xi\omega_{\Sh{A}}$. Let $\widetilde{HW}$ be the $\Sh{O}_L \otimes \Sh{O}_{\mathfrak{IG}_n}$ invertible ideal generated locally by lifts of the Hasse--Witt matrix defining $V^* \colon \omega_{\Sh{A}} \to \omega_{\Sh{A}^{(p)}}$ on the special fibre. There exists an integral model $p\Sh{H}_{\Sh{A}} \subset \Sh{H}^{\sharp}_{\Sh{A}} \subset \Sh{H}_{\Sh{A}}$ of $\Sh{H}_{\Sh{A}}$ that sits in the following exact sequence induced by the Hodge filtration.
    \[
    0 \to \Omega_{\Sh{A}} \to \Sh{H}^{\sharp}_{\Sh{A}} \to \xi\widetilde{HW}\cdot\omega^{\vee}_{\Sh{A}} \to 0.
    \]
    Moreover, the short exact sequence above admits a canonical $\Sh{O}_L$-linear splitting modulo $p\Hdg{-p^2}$, and we denote by $\Sh{Q}$ the kernel of the retraction $\psi \colon \Sh{H}^{\sharp}_{\Sh{A}}/p\Hdg{-p^2} \to \Omega_{\Sh{A}}/p\Hdg{-p^2}$.
\end{enumerate}
\end{prop}

\begin{proof}
    This is the content of \cite[\S4.1]{twisted_tripleL_finite_slope}.
\end{proof}

Let $\beta_n = p^n\Hdg{-\frac{p^n}{p-1}}$, and $\eta = p\Hdg{-p^2}$.
Let $\mathbb{V}(\Sh{H}^{\sharp}_{\Sh{A}}) \to \mathfrak{IG}_n$ be the geometric vector bundle parametrising $\Sh{O}_L$-linear functionals on $\Sh{H}^{\sharp}_{\Sh{A}}$, i.e.
\[
\mathbb{V}^{\Sh{O}_L}(\Sh{H}^{\sharp}_{\Sh{A}})(\gamma \colon \Spf{R'} \to \mathfrak{IG}_n) := \Hom_{\Sh{O}_L \otimes R'}(\gamma^*\Sh{H}^{\sharp}_{\Sh{A}}, \Sh{O}_L \otimes R') = \prod_{\sigma \in \Sigma_L} \Hom_{R'}(\gamma^* \Sh{H}^{\sharp}_{\Sh{A},\sigma}, R') = \prod_{\sigma \in \Sigma_L} \mathbb{V}(\Sh{H}^{\sharp}_{\Sh{A},\sigma}).
\]
Here $\mathbb{V}(\Sh{H}^{\sharp}_{\Sh{A},\sigma})$ is the usual geometric vector bundle associated with a locally free $\Sh{O}_{\mathfrak{IG}_n}$-module. Let $\V{\Hs{,\sigma}}{s_{\sigma}}{\Sh{Q}_{\sigma}}$ be the subfunctor of $\mathbb{V}(\Hs{,\sigma})$ defined by
\begin{align*}
\V{\Hs{,\sigma}}{s_{\sigma}}{\Sh{Q}_{\sigma}}(\gamma \colon \Spf{R'}\to \mathfrak{IG}_n) := \left\{ f \in \mathbb{V}(\Hs{,\sigma}) \, | \, (f \text{ mod } \beta_n)(\gamma^*s_{\sigma}) = 1, (f \text{ mod }\eta)(\gamma^*\Sh{Q}_{\sigma}) = 0 \right\}.
\end{align*}

\begin{lemma}\label{L2013}
    The functor $\V{\Hs{,\sigma}}{s_{\sigma}}{\Sh{Q}_{\sigma}}$ is representable.
    Let $\Spf{R}' \subset \mathfrak{IG}_n$ be a Zariski open where both $\Omega_{\Sh{A},\sigma}$ and $\Hs{,\sigma}$ are trivial, with a basis $X_{\sigma},Y_{\sigma}$ such that $X_{\sigma}$ is a lift of $s_{\sigma}$ and $Y_{\sigma}$ is a lift of a generator of $\Sh{Q}_{\sigma}$. Then we have a fibre diagram as follows.
\[\begin{tikzcd}[column sep=huge]
	{\mathfrak{IG}_n} & {\mathbb{V}(\Hs{,\sigma})} & {\V{\Hs{,\sigma}}{s_{\sigma}}{\Sh{Q}_{\sigma}}} \\
	{\Spf{R'}} & {\Spf{R'\langle X_{\sigma},Y_{\sigma}\rangle}} & {\Spf{R'}\langle Z_{\sigma}, W_{\sigma}\rangle}
	\arrow[from=1-2, to=1-1]
	\arrow[from=1-3, to=1-2]
	\arrow[from=2-1, to=1-1]
	\arrow[from=2-2, to=1-2]
	\arrow[from=2-2, to=2-1]
	\arrow[from=2-3, to=1-3]
	\arrow["{X_{\sigma}\mapsto 1+\beta_nZ_{\sigma}}", from=2-3, to=2-2]
    \arrow["{Y_{\sigma}\mapsto \eta W_{\sigma}}"', from=2-3, to=2-2]
\end{tikzcd}\]
\end{lemma}
\begin{proof}
    \cite[\S4.2,4.3]{twisted_tripleL_finite_slope}.
\end{proof}
Let $\mathbb{V}^{\Sh{O}_L}_0(\Hs{}, s, \Sh{Q}) := \prod_{\sigma\in \Sigma_L} \V{\Hs{,\sigma}}{s_{\sigma}}{\Sh{Q}_{\sigma}}$. Let $\rho \colon \mathbb{V}^{\Sh{O}_L}_0({\Hs{}},s,\Sh{Q}) \to \bar{\mathfrak{X}}^{G_L}_{\nu}$. Let $k \colon (\Sh{O}_L \otimes \Z_p)^{\times} \to R^{\times}$ be an $m$-analytic character, i.e. $k = \chi\cdot k^0$ where $\chi = k_{|(\Sh{O}_L \otimes \Z_p)^{\times}_{\text{tor}}}$ is the finite part of $k$, and $k^0$ is analytic on $1+p^m\Res_{\Sh{O}_L/\Z}(\Z_p)$. By analyticity, $k^0$ extends to a character $k^0 \colon (\Sh{O}_L \otimes \Z_p)^{\times}\cdot (1 + p^m\Res_{\Sh{O}_L/\Z}\Ga) \to \Gm$ where $\Ga, \Gm$ are sheaves on the category of $p$-adically complete formal schemes over $R$. For this fixed $m$, choose $\nu$ and $n$ large enough such that $1+\beta_n\Res_{\Sh{O}_L/\Z}\Ga \subset 1 + p^m\Res_{\Sh{O}_L/\Z}\Ga$. 

The natural action of $(\Sh{O}_L \otimes \Z_p)^{\times}$ on the finite \'etale cover $\Sh{IG}_n \to \bar{\Sh{X}}^{G_L}_{\nu}$ extends by the universal property of relative normalisation to an action on the formal model $\mathfrak{IG}_n \to \bar{\mathfrak{X}}^{G_L}_{\nu}$. Moreover, there is an action of $\mathfrak{T}^{\mathrm{ext}} := (\Sh{O}_L \otimes \Z_p)^{\times}\cdot (1 + \beta_n\Res_{\Sh{O}_L/\Z}\Ga)$ on $\rho \colon \mathbb{V}^{\Sh{O}_L}_0(\Hs{},s,\Sh{Q}) \to \bar{\mathfrak{X}}^{G_L}_{\nu}$ compatible with the above action \cite[\S4.3.1]{twisted_tripleL_finite_slope}.

\begin{definition}
    With notation as above, we define $\mathbb{W}_{k^0} := \rho_* \Sh{O}_{\mathbb{V}^{\Sh{O}_L}_0(\Hs{},s,\Sh{Q})}[k^0]$ as the functions $f \in \rho_* \Sh{O}_{\mathbb{V}^{\Sh{O}_L}_0(\Hs{},s,\Sh{Q})}$ that transform as $\lambda \ast f = k^0(\lambda)f$ for the action of $\lambda \in \mathfrak{T}^{\mathrm{ext}}$.
\end{definition}

\begin{lemma}\label{L2015}
    For $\Spf{R'} \subset \mathfrak{IG}_n$ as in Lemma \ref{L2013}, 
    \[
    \mathbb{W}_{k^0}(\Spf{R}') = R'\langle \{V_{\sigma}\}_{\sigma \in \Sigma_L}\rangle \cdot (1+\beta_nZ)^{k^0}, \quad V_{\sigma} := \frac{W_{\sigma}}{1+\beta_nZ_{\sigma}};\  (1+\beta_nZ)^{k^0} := \prod_{\sigma \in \Sigma_L}(1+\beta_n Z_{\sigma})^{k^0_{\sigma}}.
    \]
\end{lemma}

\begin{proof}
    \cite[Corollary 4.37]{twisted_tripleL_finite_slope}.
\end{proof}

\begin{definition}
    For $q = 4$ if $p = 2$, and $q=p$ otherwise, let $f \colon \mathfrak{IG}_q \to \bar{\mathfrak{X}}^{G_L}_{\nu}$ be the projection. Define $\mathfrak{w}_{\chi} = f_*\Sh{O}_{\mathfrak{IG}_q}[\chi^{-1}]$. 
\end{definition}

\begin{definition}
    We define the sheaf $\mathbb{W}_k$ of nearly overconvergent (geometric) Hilbert modular forms of weight $k$ to be $\mathfrak{w}_{\chi} \hat{\otimes} \mathbb{W}_{k^0}$. Define the subsheaf of cuspidal classes as $\mathbb{W}_k(-D) := \mathbb{W}_k \hat{\otimes}\Sh{O}_{\bar{\mathfrak{X}}^{G_L}_{\nu}}(-D)$ for the boundary divisor $D$.
\end{definition}

\begin{lemma}
    The sheaf $\mathbb{W}_k$ is equipped with an increasing filtration $\{\Fil_i\}_{i\geq 0}$ by coherent sheaves which are locally free over the generic fibre, and $\mathbb{W}_k = (\varinjlim_i \Fil_i\mathbb{W}_k)^{\wedge{}}$, where $( \cdot )^{\wedge{}}$ denotes the $p$-adic completion.
\end{lemma}
\begin{proof}
    \cite[Proposition 4.50]{twisted_tripleL_finite_slope}.
\end{proof}

\begin{remark}\label{R2019}
    Over the rigid generic fibre, passing to the finite \'etale cover $\Sh{IG}_n$, Lemma \ref{L2015} makes it clear that $\mathbb{W}_k$ is a completed colimit of locally free filtered submodules, where in the notation of the lemma, $\Fil_i\mathbb{W}_k(\Sp R'[1/p]) = R'[1/p][\{V_{\sigma}\}_{\sigma \in \Sigma_L}]^{\leq i}\cdot (1+\beta_nZ)^{k^0}$ is the submodule generated by the monomials in the $V_{\sigma}$'s of total degree $\leq i$. We will use these local coordinates to describe the Gauss--Manin connection later.
\end{remark}

We now explain how to descend the sheaf $\mathbb{W}_k$ constructed above to a sheaf of arithmetic Hilbert modular forms over $\bar{\Sh{M}}^{G_L}_{\nu}$ as in Theorem \ref{L2011}. Let $(\vartheta, \mu) \in \Sh{W}^{G_L}(R[1/p])$ be a weight such that $k = 2\vartheta+\mu t_L$ is $m$-analytic as above. For $x \in F^{\times, +}_{\mathfrak{N}p}$, and $[x] \colon \bar{\mathfrak{X}}^{G_L}_{\nu} \to \bar{\mathfrak{X}}^{G_L}_{\nu}$ given by multiplying the polarisation data by $x$, the map \[[x]^*\mathbb{W}_k \xrightarrow{\mathrm{id}} \mathbb{W}_k \xrightarrow{\times x^{\vartheta}} \mathbb{W}_k\]
defines an action of $F^{\times,+}_{\mathfrak{N}p}$ on $\mathbb{W}_k$ which respects the filtration. Let $p \colon \bar{\mathfrak{X}}^{G_L}_{\nu} \to \bar{\mathfrak{M}}^{G_L}_{\nu}$ be the projection.
\begin{definition}
    We define the sheaf of nearly overconvergent arithmetic Hilbert modular forms to be $\mathbb{W}^{G_L}_k := (p_*\mathbb{W}_k)^{F^{\times, +}_{\mathfrak{N}p}}$, i.e. the invariant sections for the action of $F^{\times,+}_{\mathfrak{N}p}$. Define $\mathfrak{w}_k^{G_L} := \Fil_0\mathbb{W}^{G_L}_k$ to be the sheaf of overconvergent arithmetic Hilbert modular forms. Define analogously the cuspidal variants.
\end{definition}

\begin{remark}
    We note that the notation $\mathbb{W}_k^{G_L}$ suppresses the information of the weight $(\vartheta, \mu)$, and in particular the fact that the infinity type of the central character is $-\mu t_L$. We regret this choice of notation, however it was made to stay consistent with the notation of \cite{twisted_tripleL_finite_slope} so that readers referring to that work don't get confused.
\end{remark}

\begin{definition}\label{D2021}
    We call a weight $k \colon (\Sh{O}_L \otimes \Z_p)^{\times} \to R^{\times}$ \textbf{good} if $k = \chi_k\cdot k^0$ where $\chi_k = k_{|(\Sh{O}_L \otimes \Z_p)^{\times}_{\mathrm{tor}}}$ is the finite part of $k$, and $k^0$ factors as
    \[
    k^0 \colon (\Sh{O}_L \otimes \Z_p)^{\times} \to (\Sh{O}_L \otimes \Sh{O}_K)^{\times} \simeq \prod_{\sigma \in \Sigma_L} \Sh{O}_K^{\times} \xrightarrow{k^0_{\sigma}} R^{\times} ,
    \]
    with $k^0_{\sigma}(t) = \exp(u_{k,\sigma}\log{t})$ for some $u_{k,\sigma} \in R$ for all $\sigma$.
\end{definition}

We recall that the Gauss--Manin connection $\nabla \colon \Sh{H}_{\Sh{A}} \to \Sh{H}_{\Sh{A}} \otimes \Omega^1_{\bar{X}^{G_L}}(\log(\text{cusps}))$ induces the Kodaira--Spencer isomorphism $\Omega^1_{\bar{X}^{G_L}}(\log(\text{cusps})) \simeq \oplus_{\sigma \in \Sigma_L} (\omega_{\Sh{A},\sigma}^2 \otimes \wedge^2\Sh{H}_{\Sh{A},\sigma}^{-1})$ which is equivariant for the action of $F^{\times,+}_{\mathfrak{N}p}$ on the polarisation.

\begin{theorem}\label{T2022}
    Let $p \neq 2$ and $k,r \colon (\Sh{O}_L \otimes \Z_p)^{\times} \to R^{\times}$ be good weights. Let $\mathbf{g}^{[\Sh{P}]} \in \mathbb{W}^{G_L}_k(-D)$ be an arithmetic overconvergent cuspform of weight $(\vartheta, \mu)$ with $2\vartheta +\mu t_L = k$, that is depleted at all primes above $p$. Then $\nabla^r(\mathbf{g}^{[\Sh{P}]}) \in \mathbb{W}^{G_L}_{k+2r}(-D)$ is a cuspform of weight $(\vartheta+r, \mu)$ after possibly shrinking the radius of overconvergence, i.e. as a section of $\mathbb{W}^{G_L}_{k+2r}(-D)$ over $\bar{\Sh{M}}^{G_L}_{\nu}$ for $\nu \gg 0$.

    Moreover, the same is true for $p=2$ if writing $r = \chi_r\cdot r^0$ as in Definition \ref{D2021}, $\chi_r$ is trivial on the $2$-torsion, and $r^0 = \prod_{\sigma} r^0_{\sigma}$ with $r^0_{\sigma}(t) = \exp(u_{r,\sigma}\log{t})$ for $u_{r,\sigma} \in 4R$.
\end{theorem}

\begin{proof}
    This is the content of \cite[Proposition 5.18]{twisted_tripleL_finite_slope} where the result is proved at the level of geometric Hilbert modular forms for the analytic part of the weight $r$. We briefly explain how to deduce the theorem from that proposition. 
    
    Firstly, the Gauss--Manin connection $\nabla \colon \mathbb{W}_k(-D) \to \mathbb{W}_k(-D) \hat{\otimes} \Omega^1_{\bar{\Sh{X}}^{G_L}_{\nu}}(\log(D))$ can be decomposed as a sum $\oplus_{\sigma \in \Sigma_L} \nabla_{\sigma}$ using the Kodaira--Spencer isomorphism, and the resulting map that we get taking the inclusion of $\omega_{\Sh{A},\sigma}^2 \subset \Sym^2\Sh{H}_{\Sh{A},\sigma}$ induces $\nabla_{\sigma} \colon \mathbb{W}_k(-D) \to \mathbb{W}_{k+2\sigma}(-D)$, which one proves can be iterated to \emph{good} weights. Furthermore, furnishing $\mathbb{W}_k$ (resp. $\mathbb{W}_{k+2\sigma}$) with the action of $F^{\times,+}_{\mathfrak{N}p}$ given by $x \ast f = x^{\vartheta}[x]^*(f)$ (resp. $x \ast f = x^{\vartheta+\sigma}[x]^*(f)$) makes $\nabla_{\sigma}$ equivariant for the action. This proves that $\nabla^r$ descends to arithmetic cuspforms for (the analytic part of) classical integer weights $r$, and the result follows for general weights by density of classical integer weights in the complete $\Sh{O}_K\llbracket (\Sh{O}_L \otimes \Z_p)^{\times}\rrbracket$ flat algebra $R$.

    Secondly, under the assumption on $r$, the finite part $\chi_r$ agrees with the finite part of some classical weight $m \in \N[\Sigma_L]$ for $p\neq 2$ (resp. $m \in 4\N[\Sigma_L]$ if $p=2$). But then writing $m = \chi_m \cdot m^0$ as a product of its finite part and analytic part, $\nabla^{\chi_m} = \nabla^{\chi_r}$ can be defined as $\nabla^m\circ \nabla^{-m^0}$ since $m^0$ now satisfies the conditions mentioned in the theorem.
\end{proof}

\begin{assumption} \label{assumption: weight r}
Let $g, f$ be as above with weights $(v,n)$ and $(w,m)$ respectively that are $F$-dominated. Let $K$ be a finite extension of $\Q_p$ containing all the Hecke eigenvalues of $g$ and $f$, and such that $L$ is split in $K$. Assume $g, f$ have finite slope $\leq a$ at a rational prime $p$ not dividing $N_gN_f$, and unramified in $L$. Let $\breve{g}, \breve{f}$ be as in Lemma \ref{L206}, which by \cite[Proposition 6]{miyake} are $K$-linear combinations of the oldforms coming from $g$ and $f$ respectively. Let $\omega_\mathbf{g}, \omega_\mathbf{f}$ be Coleman families of weight $(v_g, n_g)$ and $(w_f,2n_g)$ respectively, deforming (fixed) slope $\leq a$ $p$-stabilisations of $g$ and $f$. We suppose that the finite part of the central character of $\zeta^*\omega_\mathbf{g} \cdot \omega_\mathbf{f}$ is trivial. We note that diagonal restriction can be defined for (nearly) overconvergent families of modular forms using the formalism of vector bundles with marked sections; see \cite[Lemma 4.25, 7.7]{twisted_tripleL_finite_slope} for the relevant lemmas. Let $\breve{\omega}_\mathbf{g}, \breve{\omega}_\mathbf{f}$ be the overconvergent families of level $K_{11}(\mathfrak{A}\Sh{O}_L)$ and $K_{11}(\mathfrak{A})$ respectively, obtained from $\omega_\mathbf{g}$ and $\omega_\mathbf{f}$ using the same linear combinations of oldforms as with $\breve{g}$ and $\breve{f}$. Therefore $\breve{\omega}_\mathbf{g}$ and $\breve{\omega}_\mathbf{f}$ interpolate $p$-stabilisations of $\breve{g}$ and $\breve{f}$. Let $\breve{\omega}_\mathbf{f}^{\ast} = \breve{\omega}_\mathbf{f}|w_{\mathfrak{A}}$ where $w_{\mathfrak{A}}$ is the Atkin-Lehner involution defined in \cite[\S2]{twisted_tripleL_finite_slope} (see the discussion after Definition 2.9 in loc. cit.). Define similarly $\omega_\mathbf{f}^* = \omega_\mathbf{f} |w_{\mathfrak{A}}$.  Let $\Lambda_g$ and $\Lambda_f$ be the $p$-adically complete $K$-algebras with connected spectrum, flat over $\Sh{W}^{G_L}$ and $\Sh{W}^{G_F}$ respectively, such that $(v_g, n_g) \in \Sh{W}^{G_L}(\Lambda_g)$ and $(w_f, 2n_g) \in \Sh{W}^{G_F}(\Lambda_f)$. Let $k_g = 2v_g + n_gt_L$, and $k_f = 2w_f+2n_gt_F$. We assume that the continuous character $(w_f - {v_g}_{|F}) \colon (\Sh{O}_F \otimes \Z_p)^{\times} \to (\Lambda_g \hat{\otimes}_K \Lambda_f)^{\times}$  is \emph{good}, and satisfies the condition of Theorem \ref{T2022} for $p = 2$. We note that the locus of such weights is open and contains all classical weights (resp. all classical weights divisible by 4) for $p \neq 2$ (resp. for $p=2$). Choose any section $\Sigma_F \to \Sigma_L$ to the restriction map. Let $r = \chi_r\cdot r^0 \in (\Sh{O}_L \otimes \Z_p)^{\times} \to (\Lambda_g \hat{\otimes}_K \Lambda_f)^{\times}$ be a weight such that $\chi_r$ lifts the finite part of $(w_f - {v_g}_{|F})$, and $r^0$ lifts the torsion-free part along the chosen section. Hence the following diagram is commutative.
\[
\begin{tikzcd}
	{\Sp(\Lambda_g \hat{\otimes}_{K}\Lambda_f)} & {\Sh{W}^{G_L}} \\
	{\Sp(\Lambda_f)} & {\Sh{W}^{G_F}}
	\arrow["{(w_f, 2n_g)}"', from=2-1, to=2-2]
	\arrow["{(v_g+r, n_g)}", from=1-1, to=1-2]
	\arrow["{(v,n) \mapsto (v_{|F}, 2n)}", from=1-2, to=2-2]
    \arrow[from=1-1, to=2-1]
\end{tikzcd}
\]
We assume $k_g$ is \emph{good} too.

\end{assumption}

Let $\Lambda_{g,f} = \Lambda_g \hat{\otimes}_{K} \Lambda_f$. For sufficiently large $\ell$, the diagonal embedding of the Shimura variety for $G_F$ in the Shimura variety for $G_L$ restricts to a map of strict neighbourhoods of the ordinary loci as follows. 
\[
\begin{tikzcd}
	{\bar{\Sh{M}}^{G_F}_{\ell}} && {\bar{\Sh{M}}^{G_L}_{\ell'}} \\
	& {\Sp{\Lambda_{g,f}}}
	\arrow["\zeta", from=1-1, to=1-3]
	\arrow[from=1-1, to=2-2]
	\arrow[from=1-3, to=2-2]
\end{tikzcd}
\]
There is a Banach sheaf $\mathbb{W}^{G_F}_{k_f}(-D)$ over $\bar{\Sh{M}}^{G_F}_{\ell}$ of nearly overconvergent cuspforms for $G_F$, constructed similarly as $\mathbb{W}^{G_L}_{k_g+2r}$. Moreover, the trace map and the functoriality of vector bundles with marked sections induces a diagonal restriction map 
\[
\zeta^*\mathbb{W}^{G_L}_{k_g+2r} \xrightarrow{\zeta^*} \mathbb{W}^{G_F}_{k_f}.
\]
This allows $\zeta^*\nabla^r(\breve{\omega}^{[\Sh{P}]}_\mathbf{g})$ to be viewed as a class in $H^d_{\dR}(\bar{\Sh{M}}^{G_F}_{\ell}, \mathbb{W}^{G_F}_{k_f-2t_F}(-D)^{\bullet})$, where $\mathbb{W}^{G_F}_{k_f}(-D)^{\bullet}$ is the de Rham complex constructed using the Gauss--Manin connection \cite[\S6.4]{twisted_tripleL_finite_slope}. Using compactness of the $U$-operator, we can define a finite slope projector $e^{\leq a} \colon H^d_{\dR}(\bar{\Sh{M}}^{G_F}_{\ell}, \mathbb{W}^{G_F}_{k_f-2t_F}(-D)^{\bullet}) \to H^d_{\dR}(\bar{\Sh{M}}^{G_F}_{\ell}, \mathbb{W}^{G_F}_{k_f-2t_F}(-D)^{\bullet})^{\leq a}$ to the slope $\leq a$ subspace.

The Banach sheaf $\mathbb{W}^{G_F}_{k_f-2t_F}$ is a completed colimit of filtered subsheaves $\{\Fil_n\mathbb{W}^{G_F}_{k_f-2t_F}\}_{n\geq 0}$ which are locally free. The filtration is increasing and the Gauss--Manin connection satisfies Griffiths' transversality with respect to this filtration \cite[Corollary 5.6]{twisted_tripleL_finite_slope}. Hence we get another de Rham complex $\Fil_n^{\bullet}\mathbb{W}^{G_F}_{k_f-2t_F}(-D)$ as follows
\[
\Fil_n\mathbb{W}^{G_F}_{k_f-2t_F}(-D) \xrightarrow{\nabla} \Fil_{n+1}\mathbb{W}^{G_F}_{k_f-2t_F}(-D)\hat{\otimes}\Omega^1(\log(D))_{\bar{\Sh{M}}^{G_F}_{\ell}/\Lambda_{g,f}} \xrightarrow{\nabla} \dots \xrightarrow{\nabla} \Fil_{n+d}\mathbb{W}^{G_F}_{k_f-2t_F} \hat{\otimes} \Omega^d(\log(D))_{\bar{\Sh{M}}^{G_F}_{\ell}/\Lambda_{g,f}}.
\]

\begin{lemma}
    There exists $n \gg 0$ (depending on $a$) such that the map 
    \[
    H^d_{\dR}(\bar{\Sh{M}}^{G_F}_{\ell}, \Fil_n^{\bullet}\mathbb{W}^{G_F}_{k_f-2t_F}(-D))^{\leq a} \to H^d_{\dR}(\bar{\Sh{M}}^{G_F}_{\ell}, \mathbb{W}^{G_F}_{k_f-2t_F}(-D)^{\bullet})^{\leq a}
    \]
    is an isomorphism.
\end{lemma}

\begin{proof}
    \cite[Lemma 6.20]{twisted_tripleL_finite_slope}.
\end{proof}

Recall $k^0_{g,\sigma} = \exp(u_{k,\sigma}\log{(\cdot)})$ and $r^0_{\sigma} = \exp(u_{r,\sigma}\log(\cdot))$ for all $\sigma \in \Sigma_L$. Let $\rho_{\tau} = u_{k,\sigma} + u_{r,\sigma} + u_{k,\bar{\sigma}} + u_{r,\bar{\sigma}}$ for $\tau \in \Sigma_F$ and $\sigma, \bar{\sigma}$ being the two lifts of $\tau$. Note that one of $u_{r,\sigma}$ and $u_{r,\bar{\sigma}}$ is $0$ by our definition of $r^0$. Let $\lambda = \prod_{\tau \in \Sigma_F}\prod_{i=0}^n(\rho_{\tau} - i)$ for $n$ as in the above lemma. 

\begin{lemma}
    Let $\mathfrak{w}^{G_F}_{k_f} = \Fil_0\mathbb{W}^{G_F}_{k_f}$ be the sheaf of weight $k_f$ overconvergent modular forms. We have an isomorphism 
    \begin{align*}
    H^{\dagger} \colon H^d_{\dR}(\bar{\Sh{M}}^{G_F}_{\ell}, \mathbb{W}^{G_F}_{k_f-2t_F}(-D)^{\bullet})^{\leq a} \otimes \Lambda_{g,f}[\lambda^{-1}] &\xrightarrow{\sim} H^d_{\dR}(\bar{\Sh{M}}^{G_F}_{\ell}, \Fil_n^{\bullet}\mathbb{W}^{G_F}_{k_f-2t_F}(-D))^{\leq a} \otimes \Lambda_{g,f}[\lambda^{-1}] \\
    &\xrightarrow{\sim} H^0(\bar{\Sh{M}}^{G_F}_{\ell}, \mathfrak{w}^{G_F}_{k_f}(-D))^{\leq a} \otimes \Lambda_{g,f}[\lambda^{-1}].
    \end{align*}
\end{lemma}

\begin{proof}
    \cite[Lemma 6.19]{twisted_tripleL_finite_slope}.
\end{proof}

\begin{cor}
    Denoting by $H^{\dagger,\leq a}$ the composition $H^{\dagger} \circ e^{\leq a}$, \[
    H^{\dagger,\leq a}(\zeta^*\nabla^r\breve{\omega}_\mathbf{g}^{[\Sh{P}]}) \in H^0(\bar{\Sh{M}}^{G_F}_{\ell}, \mathfrak{w}^{G_F}_{k_f}(-D))^{\leq a} \otimes \Lambda_{g,f}[\lambda^{-1}].
    \]
\end{cor}

As a result using the recipe of \cite[Lemma 2.19]{DR} to define a $p$-adic pairing $\langle \phantom{e}, \phantom{e} \rangle$ we can define the $p$-adic twisted triple product $L$-function as follows.

\begin{definition} \label{definition: twisted triple product}
    The $p$-adic twisted triple product $L$-function associated with the Coleman families $\breve{\omega}_\mathbf{g}, \breve{\omega}_\mathbf{f}$ is defined to be
    \[
    \mathscr{L}_p(\breve{\omega}_\mathbf{g},\breve{\omega}_\mathbf{f}) = \frac{\langle H^{\dagger,\leq a}(\zeta^*\nabla^r\breve{\omega}_\mathbf{g}^{[\Sh{P}]}), \breve{\omega}_\mathbf{f}^*\rangle}{\langle \omega_\mathbf{f}^*, \omega_\mathbf{f}^*\rangle}.
    \]
\end{definition}

\subsection{Interpolation formula}
Let now $(x,c) \in \N[\Sigma_L]\times \N$ and $(y,2c) \in \N[\Sigma_F] \times \N$ be two classical weights given by a closed point of $\Sp{\Lambda_{g,f}[\lambda^{-1}]}$. Thus, $r$ specialises to a weight $t \in \N[\Sigma_L]$ such that $(x+t)_{|F} = y$. Assume further that the weights are large enough compared to the slope bound $a$ (for example, as in \cite[Theorem 1]{TianXiao}) such that the specialisations of $\breve{\omega}_\mathbf{g}$ and $\breve{\omega}_\mathbf{f}$ respectively at $(x,c)$ and $(y,2c)$ are slope $\leq a$ $p$-stabilisations of classical eigenforms $\breve{g}_x \in S^{G_L}(K_{11}(\mathfrak{A}\Sh{O}_L),(x,c), K)$ and $\breve{f}_y \in S^{G_F}(K_{11}(\mathfrak{A}), (y,2c), K)$ whose Hecke polynomial for $T_0({\mathfrak{p}})$  for any prime $\mathfrak{p}$ above $p$ in $L$ (resp. $F$) has roots with distinct slopes.

\begin{remark}
    For this article, Hecke operator $T_0({\mathfrak{p}})$ for any prime $\mathfrak{p}|p$ means the ``optimally integral" Hecke operator which on  automorphic forms of weight $(x,c) \in \N[\Sigma_L] \times \N$ and level $K$ act as \[ T_0({\mathfrak{p}}) = \mathfrak{p}^{x-t_L}\left[K\left(\begin{smallmatrix}
        \varpi_{\mathfrak{p}} & 0 \\ 0 & 1
    \end{smallmatrix}\right)K\right].\]
    Here $\mathfrak{p}^{x-t_L} = p^{\sum_{\sigma|\mathfrak{p}}(x_{\sigma}-\sigma)}$.
\end{remark}

Let $\alpha_{g, \ell}, \beta_{g, \ell}$ be the $T_0({\ell})$-eigenvalues of $\breve{g}_x$ for any prime $\ell \nmid \mathfrak{A}\Sh{O}_L$ in $\Sh{O}_L$. Let $\alpha^*_{f,\ell}, \beta^*_{f,\ell}$ be $T_0({\ell})$-eigenvalues of $\breve{f}^*_y$ for any prime $\ell \nmid \mathfrak{A}$ in $\Sh{O}_F$. We also consider the Hecke operator $T_0(p) = \prod_{\mathfrak{p}|p} T_0({\mathfrak{p}})$ acting on $\breve{g}_x$ (resp. on $\breve{f}^*_y)$ where the product ranges over all the primes above $p$ in $L$ (resp. in $F$). Let $\alpha_{g,p}, \beta_{g,p}$ (resp. $\alpha^*_{f,p}, \beta^*_{f,p}$) be the $T_0(p)$-eigenvalues of $\breve{g}_x$ (resp. $\breve{f}^*_y$). Without loss of generality, we assume that both $\alpha_{g, p}$ and $\alpha^*_{f,p}$ have $p$-adic valuation $\leq a$.

We define appropriate Euler factors as follows:
\begin{itemize}
    \item $\mathscr{E}(\breve{f}^*_y) = (1 - \beta^*_{f,p}{\alpha^{*,-1}_{f,p}})$.
    \item For any prime $\mathfrak{p}|p$ in $\Sh{O}_F$ that stays inert in $L$, let $\mathscr{E}_{\mathfrak{p}}(\breve{g}_x, \breve{f}^*_y) = (1 - \mathfrak{p}^{t_{|F}}\alpha_{g,\mathfrak{p}}{\alpha^{*,-1}_{f,\mathfrak{p}}})(1 - \mathfrak{p}^{t_{|F}}\beta_{g,\mathfrak{p}}{\alpha^{*,-1}_{f,\mathfrak{p}}})$. Here writing $t_{|F} = \sum_{\sigma \in \Sigma_F} t_{\sigma}\sigma$, we define $\mathfrak{p}^{t_{|F}} = p^{\sum_{\sigma|\mathfrak{p}}t_{\sigma}}$.
    \item For any prime $\mathfrak{p}|p$ in $\Sh{O}_F$ that splits as $\mathfrak{p} = \mathfrak{p}_1\mathfrak{p}_2$ in $\Sh{O}_L$, let $\alpha_i = \alpha_{g,\mathfrak{p}_i}$, and $\beta_i = \beta_{g, \mathfrak{p}_i}$ for $i \in \{1,2\}$. Let 
    \[
    \mathscr{E}_{\mathfrak{p}}(\breve{g}_x, \breve{f}^*_y) = \prod_{\bullet, \diamond \in \{\alpha,\beta\}} (1 - \mathfrak{p}^{t_{|F}}\bullet_1\diamond_2\alpha^{*,-1}_{f,\mathfrak{p}}); \qquad \mathscr{E}_{0,\mathfrak{p}}(\breve{g}_x,\breve{f}^*_y) = (1 - \mathfrak{p}^{2t_{|F}}\alpha_1\alpha_2\beta_1\beta_2\alpha^{*,-2}_{f,\mathfrak{p}}).
    \]
\end{itemize}

\begin{theorem} \label{Theorem: interpolation formula}
    With the assumptions above, the evaluation of $\mathscr{L}_p(\breve{\omega}_\mathbf{g}, \breve{\omega}_\mathbf{f})$ at the $F$-dominated pair of weights $\left((x,c), (y,2c)\right)$ is given by
    \[
    \mathscr{L}_p(\breve{\omega}_\mathbf{g}, \breve{\omega}_\mathbf{f})\left((x,c), (y,2c)\right) = \frac{1}{\mathscr{E}(\breve{f}^*_y)}\left(\prod_{\mathfrak{p} \emph{\text{ inert}}} \mathscr{E}_{\mathfrak{p}}(\breve{g}_x, \breve{f}^*_y) \prod_{\mathfrak{p} \emph{\text{ split}}} \frac{\mathscr{E}_{\mathfrak{p}}(\breve{g}_x, \breve{f}^*_y)}{\mathscr{E}_{0,\mathfrak{p}}(\breve{g}_x,\breve{f}^*_y)} \right) \times \frac{\langle\zeta^*(\delta^t \breve{g}_x), \breve{f}^*_y\rangle}{\langle \breve{f}^*_y,\breve{f}^*_y\rangle}.
    \]
\end{theorem}

\begin{proof}
    This follows from computations similar to \cite[\S3.4]{blanco-chacón_fornea_2020}. We remark that in loc. cit. the Euler factors look different because the authors use Hecke eigenvalues for the non-normalised Hecke operators $\mathfrak{p}^{t_L-x}T_0({\mathfrak{p}})$ for any prime $\mathfrak{p}|p$ in $\Sh{O}_L$ (and similarly for $\Sh{O}_F$).
\end{proof}

\begin{remark}
    In \cite[Theorem 7.27]{twisted_tripleL_finite_slope} we proved a similar formula under the assumption that $\omega_{\mathbf{g}}, \omega_{\mathbf{f}}$ are general overconvergent families whose specialisations at the weights $(x,c), (y,2c)$ as above were assumed to be classical $T_0(p)$-eigenforms, instead of $U$-eigenforms. Hence we get slightly different Euler factors in that setting accounting for the two $p$-stabilisations of the $T_0(p)$-eigenform $f^*_y$.
\end{remark}
\newcommand{\calO}{\mathcal{O}}
\section{The generalized Hirzebruch--Zagier cycles and the Abel--Jacobi maps}
\label{section: GZ-formula}

From now on, we will mainly restrict ourselves to the case where $F = \mathbb{Q}$ and $L/ \mathbb{Q}$ is a real quadratic extension with two real embeddings denoted by $\sigma_1, \sigma_2$.

\subsection{Several formulae on \texorpdfstring{$q$}{}-expansions}

We write $\nabla_i = \nabla(\sigma_i)$ and let $d_i$ be the derivation corresponding to $\sigma_i$.
Namely, around a Tate object as defined in \cite[\S5.3]{twisted_tripleL_finite_slope}, $d_i$ acts on the $q$-expansions via
\[ d_i (\sum_\beta a_\beta q^\beta ) = \sum \sigma_i(\beta) a_\beta q^\beta ,\]
where $\sum_\beta a_\beta q^\beta$ belongs to the ring of $q$-expansions $R:= \Lambda^0_I  ( \!( \mathfrak{a} \mathfrak{b}, S)\!)$ corresponding to the Tate object.
Then, the Gauss--Manin connection can be described by the following formula:
\begin{align}
\begin{split}
    \nabla_{i} \left ( a \cdot \prod_\sigma V_\sigma^{h_\sigma} (1+ \beta_n Z)^k \right ) &= d_{i} (a)  \prod_\sigma V_\sigma^{h_\sigma} \cdot (1+ \beta_n Z)^{k+2 \sigma_i} \\
    &+ p (u_{k_{\sigma_i}} - h_{\sigma_i}) V_{\sigma_i} \prod_\sigma V_\sigma^{h_\sigma} \cdot (1+ \beta_n Z)^{k+2 \sigma_i}.
\end{split}
\end{align}

For a $\Sh{P}$-depleted form $\mathbf{g}^{[\Sh{P}]}$ with the $q$-expansion 
$\mathbf{g}^{[\Sh{P}]} (q) \cdot (1 +\beta_n Z)^{k}$ 
and a weight $r$ as in Assumption \ref{assumption: weight r}, the $q$-expansion of $\nabla^r g^{[\Sh{P}]}$ is then given by
\begin{equation} \label{equation: iteration}
    \sum_{j=0}^{\infty} p^j \binom{u_{r}}{j} \prod_{i =0}^{j-1} (u_{k_{\sigma_1}} +u_{r} -1-i) d_{1}^{{r}-j} \mathbf{g}^{[\mathcal{P}]}(q) \cdot V^j_{\sigma_1} (1+\beta_n Z)^{k +2 r }.
\end{equation}
Here we recall that the weight $r$ was defined as $\chi_r\cdot r^0$, where $r^0$ is the analytic part of the character and contributes to one embedding, which we can and will assume without loss of generality to be $\sigma_1$. 
The weight $\chi_r$ lifts a character of $(\Z_p^{\times})_{\mathrm{tor}}$ to a character of $(\calO_F \otimes \Z_p)^{\times}_{\mathrm{tor}}$ which is trivial on the $2$-torsion if $p=2$. We can assume without loss of generality that $\chi_r$ is the finite part of a classical weight $(m,0) \in \Z[\Sigma_L]$. As an upshot, all classical specialisations of $r$ will only contribute to the $\sigma_1$ direction.
By abuse of notation, we will usually replace $u_{k_{\sigma_1}}$ by $k_{\sigma_1}$ in the iteration formula and similarly for $r$.

\begin{remark}
    By construction, the element $1 +\beta_n Z_{\sigma_i}$ can be identified with $\omega_{\can, i}$ and $V_{\sigma_i}$ can be identified with $p^{-1}\omega_{\can, i}^{-1} \cdot \eta_{\can, i}$, where $\omega_{\can, i}, \eta_{\can, i}$ are the $\sigma_i$-part of the canonical basis of the first de Rham cohomology of the universal abelian surface over the Tate object.
    We will freely use these identifications in our computations later.
\end{remark}

We also recall that the diagonal restriction $\zeta^*$ maps an element of the form $a \cdot \prod_\sigma V_\sigma^{h_\sigma} (1+ \beta_n Z)^k $ to $\zeta^*(a) \cdot V^{h_{\sigma_1}+ h_{\sigma_2}} (1+\beta_n Z)^{k_{\sigma_1} +k_{\sigma_2}}$.
As we will not need to deal with the $q$-expansion $\zeta^* (a)$, we prefer not to unwind its explicit formula.

Lastly, we recall the formula for the overconvergent projection $H^\dagger$, or rather, its restriction to $\Fil_n \mathbb{W}_k$.
Since we assume that $F = \mathbb{Q}$, we only need to consider the case of modular curves, which is studied in \cite{andreatta2021triple}.
Given any section $\gamma$ of $\Fil_n \mathbb{W}_k$ with $q$-expansion $\sum_{i \leq n} \gamma_i(q) p^iV^i (1 +\beta_n Z)^k$, the $q$-expansion of $H^{\dagger} (\gamma)$ is given by (c.f. \cite[Proposition~3.37]{andreatta2021triple})
\begin{equation} \label{equation: overconvergent projection}
    \sum_{i=0}^n (-1)^i \frac{d^i \gamma_i (q) }{(k-2-i +1) (k-2-i+2) \cdots (k-2)} \cdot (1+\beta_n Z)^k.
\end{equation}

\subsection{Specialisations at classical weights}

Let $\omega_\mathbf{g}, \omega_\mathbf{f}$ be Coleman families such that we can define the twisted triple product $p$-adic $L$-function 
\begin{equation*}
    \mathscr{L}_p(\breve{\omega}_\mathbf{g},\breve{\omega}_\mathbf{f}) = \frac{\langle H^{\dagger,\leq a}(\zeta^*\nabla^r\breve{\omega}_\mathbf{g}^{[\Sh{P}]}), \breve{\omega}_\mathbf{f}^*\rangle}{\langle \omega_\mathbf{f}^*, \omega_\mathbf{f}^*\rangle}
\end{equation*}
as in Definition \ref{definition: twisted triple product}.
For simplicity, we will drop the $\breve{\phantom{e}}$ notations throughout this section.

Let 
\begin{equation}
    \mathbf{g}(q) \cdot (1+ \beta_n Z)^{k_g} \quad (\textrm{resp.}\  \mathbf{g}^{[\mathcal{P}]}(q) \cdot (1+\beta_n Z)^{k_g} )
\end{equation}
be the $q$-expansion of $\omega_\mathbf{g}$ (resp. $\omega_\mathbf{g}^{[\mathcal{P}]}$).

Fix a balanced pair $(\ell, k )$ with $k = \ell_1 + \ell_2 -2 (s+1)$, $s \in \mathbb{Z}_{\geq 0}$.
The specialization of $\mathscr{L}_p (\omega_\mathbf{g}, \omega_\mathbf{f})$ at $(\ell, k )$ takes the form
\begin{equation*}
    \frac{ \langle H^{\dagger, \leq a} (\zeta^* \nabla^{(-s-1, 0)} g_{\ell}^{[\mathcal{P}]}), f_k^* \rangle}{\langle f_k^*, f_k^* \rangle}.
\end{equation*}

First, we compute the (polynomial) $q$-expansion of $\nabla^{(-s-1, 0)} g_{\ell}^{[\mathcal{P}]} = \nabla_1^{-s-1} g_{\ell}^{[\mathcal{P}]}$.
By formula (\ref{equation: iteration}), one gets 
\begin{align*}
     \nabla_1^{-s-1}& g_{\ell}^{[\mathcal{P}]}(q) \cdot (1+\beta_n Z)^{\ell} \\
    &= \sum_{j=0}^{\infty} p^j \binom{-s-1}{j} \prod_{i =0}^{j-1} (\ell_1 -s -2 -i) d_{1}^{-s-1-j} g_{\ell}^{[\mathcal{P}]}(q) \cdot V^j_{\sigma_1} (1+\beta_n Z)^{\ell+ 2(-s-1, 0)} \\
    &= \sum_{j=0}^{\ell_1-s-2} \binom{-s-1}{j} \prod_{i =0}^{j-1} (\ell_1 -s -2 -i) d_{1}^{-s-1-j} g_{\ell}^{[\mathcal{P}]}(q) \cdot \omega_{\can}^{\ell+ (-2s-2 -j, 0)} \eta_{\can}^{(j, 0)}.
\end{align*}

Its diagonal pull-back $\zeta^*( \nabla^{(-s-1, 0)} g_{\ell}^{[\mathcal{P}]})$ then takes the following form:
\begin{equation}
    \sum_{j=0}^{\ell_1-s-2} \binom{-s-1}{j} \prod_{i =0}^{j-1} (\ell_1 -s -2 -i) \zeta^* \left (d_{1}^{-s-1-j} g_{\ell}^{[\mathcal{P}]}(q) \right ) \cdot \omega_{\can}^{k-j} \eta_{\can}^{j}.
\end{equation}

Later, we will restrict ourselves to classical pair $(\ell, k)$ satisfying the following assumption.
\begin{assumption} \label{assumption: classical specialisation}
    We will assume that the specialisation of $\mathbf{g}$ at $\ell$ is a $p$-stabilisation of a classical form denoted by $g_\ell$.
    Similarly, we assume that the specialisation of $\mathbf{f}$ at $k$ is a slope $\leq a$ $p$-stabilisation of a classical form denoted by $f_k$.
\end{assumption}

\subsection{The Abel--Jacobi images} \label{subsection: AJ images}

In this subsection, we examine the syntomic Abel--Jacobi image of the generalized \textit{Hirzebruch--Zagier} cycle.
The arguments follow closely those in \cite[\S5]{blanco-chacón_fornea_2020} and we will give an analogue of Theorem 5.14 in \textit{loc. cit.}.

 
Let $g \in S^{G_L}(K_{11}(\mathfrak{A}\Sh{O}_L), (v, n), K)$ and $f \in S^{G_F}(K_{11}(\mathfrak{A}), (w, m), K)$ be two eigenforms with coefficients in a finite extension $K$ of $\mathbb{Q}_p$, where $p$ does not divide $\mathfrak{A} \cdot d_{L / \Q}$.
Let $\ell = (\ell_1, \ell_2) = 2v +n t_L$, and $k = 2w +m$, and assume that $k$ and $(\ell_1, \ell_2)$ are balanced.
In particular, we will assume that $k = \ell_1 + \ell_2 - 2 (s+1)$ for some $s \geq 0$.

We first need to recall several geometric objects and idempotents that give rise to the desired cohomology spaces. 
We would like to refer to \cite[\S4 \& \S5.1]{blanco-chacón_fornea_2020} for detailed definitions.
Given $g, f$ as above, we have the following objects:
\begin{enumerate}
    \item the universal elliptic curve $\mathcal{E}$ over the modular curve $\mathrm{Sh}_{K_{11}(\mathfrak{A})}(\GL{2, \mathbb{Q}} ) = Y_{11}(\mathfrak{A})$;
    \item if $k >2$, we let $W_{k-2}$ be a smooth compactification of the $(k-2)$th-fold fiber product of $\mathcal{E}$ over $\mathrm{Sh}_{K_{11}(\mathfrak{A})}(\GL{2, \mathbb{Q}} )$;
    \item there exists an idempotent $\theta_k \in \mathrm{CH}^{k-1} (W_{k-2} \times W_{k-2})$ such that $\theta_k^* H^{k-1}_{\dR}(W_{k-2} / K)$ is functorially isomorphic to the parabolic cohomology $H^1_{\mathrm{par}} (X_{11}(\mathfrak{A}),\mathcal{F}_{\GL{2, \mathbb{Q}}}^{(k, k-1)})$ with Hodge filtrations (c.f. \cite[Proposition~5.3]{blanco-chacón_fornea_2020}).
    Here, $X_{11}(\mathfrak{A})$ is the compactified modular curve and the vector bundle $\mathcal{F}_{\GL{2, \mathbb{Q}}}^{(k, k-1)}$ defined in \cite[\S4]{blanco-chacón_fornea_2020} is equal to $\Sym^{k-2} \mathcal{H}_{\mathcal{E}}$ in \S\ref{subsection: Hilbert modular forms}. Therefore, $H^1_{\mathrm{par}} (X_{11}(\mathfrak{A}),\mathcal{F}_{\GL{2, \mathbb{Q}}}^{(k, k-1)})$ is the first hypercohomology group of the complex 
    \[
    \Sym^{k-2} \Sh{H}_{\Sh{E}}(-D)) \xrightarrow{\nabla} \Sym^{k-2}\Sh{H}_{\Sh{E}}(-D) \otimes \Omega^1_{X}(\log(D)).
    \]
    \item the universal abelian scheme $\mathcal{A}$ over the Shimura variety $\mathrm{Sh}_{K_{11}(\mathfrak{A}\Sh{O}_L)}(G_L^*)$ (c.f. Remark \ref{remark: Hilbert modular scheme, G*});
    \item if $\ell > 2 t_L$, we let $U_{\ell-4}$ be a smooth compactification of the $(|\ell|-4)$-fold fiber product of $\mathcal{A}$ over $\mathrm{Sh}_{K_{11}(\mathfrak{A}\Sh{O}_L)}(G_L^*)$;
    \item there exists an idempotent $\theta_\ell \in \mathrm{CH}^{2(|\ell|-3)} (U_{\ell-4} \times U_{\ell-4})$ such that $\theta_\ell^* H^{i+|\ell| -4}_{\dR} (U_{\ell-4}/ K)$ is functorially isomorphic to the hypercohomology $\mathbb{H}^i (\mathrm{Sh}_{K_{11}(\mathfrak{A}\Sh{O}_L)} (G_L^*), \mathrm{DR}^\bullet (\mathcal{F}_{G_L^*}^{(\ell, |\ell-t_L|)} ))$ with Hodge filtrations (c.f. \cite[Proposition~5.4]{blanco-chacón_fornea_2020}). 
    Here, the sheaf $\mathcal{F}_{G_L^*}^{(\ell, |\ell-t_L|)}$ defined in \cite[\S4.2]{blanco-chacón_fornea_2020} is equal to
    $\Sym^{\ell-2 t_L} \mathcal{H}_{\mathcal{A}}$ in \S\ref{subsection: Hilbert modular forms}. 
    Similarly as in point (3), $\mathbb{H}^i (\mathrm{Sh}_{K_{11}(\mathfrak{A}\Sh{O}_L)} (G_L^*), \mathrm{DR}^\bullet (\mathcal{F}_{G_L^*}^{(\ell, |\ell-t_L|)} ))$ is the $i$-th hypercohomology group of the complex 
    \[
    \Sym^{\ell-2t_L}(-D) \xrightarrow{\nabla} \Sym^{\ell - 2t_L}(-D) \otimes \Omega^1(\log(D))_{\mathrm{Sh}^{\mathrm{tor}}_{K_{11}(\mathfrak{A}\Sh{O}_L)}(G_L^*)} \xrightarrow{\nabla} \Sym^{\ell - 2t_L}(-D) \otimes \Omega^2(\log(D))_{\mathrm{Sh}^{\mathrm{tor}}_{K_{11}(\mathfrak{A}\Sh{O}_L)}(G_L^*)}.
    \]
\end{enumerate}

Now, we recall the definition of the \textit{Hirzebruch--Zagier} cycle associated with a balanced tuple $(\ell, k)$.
Suppose first that $(\ell, k) \neq (2t_L, 2)$ and $\ell$ is not parallel. 
Let $\gamma = \frac{|\ell|+k-6}{2}$.
Consider the finite map
\begin{align*}
    \varphi: \mathcal{E}^\gamma &\longrightarrow \mathcal{A}^{|\ell|-4} \times_K \mathcal{E}^{k-2} \\
    (x; P_1, \ldots, P_\gamma) &\mapsto (\zeta(x), P_1' \otimes 1, \ldots , P_{|\ell|-4}' \otimes 1 ; x, P_{|\ell|-3}', \ldots , P_{2 \gamma}' )
\end{align*}
where we label $(P_1', \ldots , P_{2 \gamma}') = (P_1, \ldots , P_{\gamma}, P_1, \ldots, P_{\gamma})$ and $P_i' \otimes 1$ is the image of $P_i'$ under the natural map $\mathcal{E} \otimes _\mathbb{Z} \mathcal{O}_L \rightarrow \mathcal{A}$.
As explained in \cite[\S5.2.1]{blanco-chacón_fornea_2020}, 
there are smooth projective models $\mathscr{W}_\gamma, \mathscr{U}_{\ell-4}, \mathscr{W}_{k-2}$ of $W_\gamma, U_{|\ell|-4}, W_{k-2}$ respectively, over an open of $\Spec(\mathcal{O}_K)$, such that the map $\varphi$ can be extended to a morphism $\mathscr{W}_\gamma \rightarrow \mathscr{U}_{\ell-4} \times \mathscr{W}_{k-2}$ which we will still denote by $\varphi$.
In addition, the correspondences $\theta_k, \theta_\ell$ can also be extended to these smooth models by spreading out.

\begin{definition}
    Let $\ell \in \Z [\Sigma_L]$, $\ell > 2 t_L$ be a non-parallel weight and $k \in \Z_{>2}$ such that $(\ell, k)$ is a balanced tuple.
    The {Hirzebruch--Zagier} cycle $\Delta_{\ell, k} $ for weight $(\ell, k)$ is defined to be 
    \begin{equation*}
        ({\theta}_\ell, {\theta}_k)^*{\varphi}_*(\mathscr{W}_\gamma) \in \CH^{d-\gamma-1}(\mathscr{U}_{\ell-4} \times \mathscr{W}_{k-2}),
    \end{equation*}
    where $d = 2|\ell| + k-7$ is the relative dimension of the variety $\mathscr{U}_{\ell-4} \times \mathscr{W}_{k-2}$.
\end{definition}

When $\ell=2 t_L$ and $k=2$, we need a different approach as follows.
Assume that $K$ is large enough such that $U_{0/ K}$ (resp. $W_{0/K}$) is a disjoint union $U_{0/K} = \coprod U_{0, i}$ (resp. $W_{0/K} = \coprod W_{0. j}$) of geometrically connected components.
We pick a $K$-rational point $a_i \in U_{0, i}$ (resp. $b_j \in W_{0, j}$) for each component.

Set $Z = U_0 \times W_0$. 
For each pair $(i,j)$ of the indices, we define $q_{i,j}: Z \rightarrow Z$ to be the map which is the identity map on $U_{0, i} \times W_{0, j}$ and sends other components $U_{0, i'} \times W_{0, j'}$ to the point $(a_{i'}, b_{j'})$.
Similarly, one defines $q_{a_i, j} : Z \rightarrow \{a_i\}  \times W_0$, $q_{i, b_j} Z\rightarrow U_0 \times \{b_j\}$ and $q_{a_i, b_j}: Z \rightarrow \{a_i\} \times \{ b_j \}$.
One now considers the graphs $P_{i,j} := \mathrm{graph}(q_{i,j})$ and similarly $P_{a_i, j}$, $P_{i, b_j}$ and $P_{a_i, b_j}$.
These graphs define a correspondence
$$P:= \sum_{i,j} (P_{i,j} - P_{a_i,j} - P_{i, b_j} + P_{a_i, b_j})$$
in $\CH^3 (Z \times Z)$, which acts on $\CH^\bullet(Z)$ by $P_* = \pr_{2, *}(P \cdot \pr_1^*)$.
In other words, for any cycle $S \in \CH^\bullet(Z)$, we have
\begin{equation*}
    P_*(S) = \sum_{i,j} \left [ (q_{i,j})_* - (q_{a_i,j})_* - (q_{i,b_j})_* + (q_{a_i,b_j})_* \right ](S).
\end{equation*}
By an abuse of notation, we also let ${P}$ be the correspondence on $\mathscr{U}_0 \times \mathscr{W}_0$ defined over some open of $\Spec (\mathcal{O}_K)$ by spreading out.
Let $\varphi : \mathscr{W}_0 \rightarrow \mathscr{U}_0 \times \mathscr{W}_0$ be the map whose first component is the diagonal embedding and the second component is the identity map.

\begin{definition}
    The {Hirzebruch--Zagier} cycle for weight $(2t_L, 2)$ is defined to be 
    \begin{equation*}
        \Delta_{2t_L, 2} := {P}_* \circ {\varphi}_* (\mathscr{W}_0) \in \CH^2 (\mathscr{U}_0 \times \mathscr{W}_0).
    \end{equation*}
    One can check that $\Delta_{2t_L,2}$ is de Rham null-homologous (c.f. \cite[Proposition~5.7]{blanco-chacón_fornea_2020}).
\end{definition}

After introducing the cycle, we are now interested in computing the syntomic Abel--Jacobi image $\AJ_p(\Delta_{\ell, k})(\pi_1^* \omega_g \cup \pi_2^* \eta_f)$ and expressing it in terms of overconvergent modular forms.
The computations are similar to the case of triple product $p$-adic $L$-functions and we will frequently refer to \cite{DR} and \cite{tripleL_and_GZ_formula_finite_slope}.

First, we need to specify the elements $\omega_g$ and $\eta_f$. Recall from \S\ref{subsection: Hilbert modular forms} that the modular form $g$ of weight $(v,n)$ defines a section of $\omega_{\Sh{A}}^{\ell} \otimes (\wedge^2_{\Sh{O}_L} \Sh{H}_{\Sh{A}})^{-v}$, and that $\eta_{\lambda, \mathfrak{d}_L^{-1}}$ is a canonical generator of $\wedge^2_{\Sh{O}_L} \Sh{H}_{\Sh{A}}$ over $X^{G_L, \mathfrak{d}_L^{-1}} = \mathrm{Sh}_{K_{11}(\mathfrak{A}\Sh{O}_L)} (G_L^*)$. Therefore $g \cdot \eta_{\lambda, \mathfrak{d}_L^{-1}}^{v-t_L} \in \omega_{\Sh{A}}^{\ell} \otimes (\wedge^2_{\Sh{O}_L} \Sh{H}_{\Sh{A}})^{-t_L}$  defines a differential form $\omega_g$ over $\mathrm{Sh}_{K_{11}(\mathfrak{A}\Sh{O}_L)} (G_L^*)$, which can be viewed as an element in $\Fil^{|\ell|-2} \theta_\ell^* H^{|\ell|-2}_{\dR} (U_{\ell-4} /K)$.
For $\eta_f$, it is the unique element in the $f$-isotypic part of $\theta_k^* H^{k-1}_{\dR}(W_{k-2} / K)$ such that $\mathrm{Fr}_p (\eta_f) = \alpha_{f}^* \eta_f$ and for any cusp form $h$, we have $\langle \omega_h, \eta_f \rangle_{\dR} = \frac{\langle \omega_h, f^*\rangle}{ \langle f^*, f^*\rangle}$ (c.f. \cite[p.15]{tripleL_and_GZ_formula_finite_slope}). 
In what follows, we will simply write $\omega =\omega_g$ and $\eta = \eta_f$.

In the case $(\ell, k) = (2 t_L, 2)$, it follows from 
$H^1_{\fp}(\Spec (\calO_{K}), 0 ) = 0 = H^2_{\fp}(\Spec (\calO_{K}), 2 )$ that (c.f. \cite[\S~4]{reg-formula})
\begin{equation}
    \AJ_p(\Delta_{2t_L, 2})(\pi_1^* {\omega} \cup \pi_2^* \eta) = \tr_{\mathscr{W}_0, \fp}( \varphi^*( \pi_1^* \tilde{\omega} \cup \pi_2^* \tilde{\eta})) = \tr_{\mathscr{W}_0, \fp} (\zeta^*\tilde{\omega} \cup \tilde{\eta}),
\end{equation}
where $\tilde{\omega}, \tilde{\eta}$ are the lifts in the finite polynomial cohomology (c.f. \cite{fp} and the explanation below).

For the general case, we recall the following numbers (c.f. \cite[(25)]{blanco-chacón_fornea_2020}):
\begin{enumerate}
    \item the dimension of $U_{\ell-4} \times W_{k-2}$ is $d = 2 |\ell|+k-7$;
    \item the cycle $\Delta_{\ell, k}$ is of dimension $\gamma+1$, with $\gamma = \frac{|\ell| +k-6}{2} $;
    \item a positive integer $s:= \frac{|\ell|-k-2}{2}$, which satisfies the equation $|\ell| -2-s = \gamma +2$.
\end{enumerate}

For the element $\eta$, there is an isomorphism $\pr_{\fp}: H^{k-1}_{\fp}(\mathscr{W}_{k-2}, 0) \cong H^{k-1}_{\dR}(W_{k-2})$.
We hence let $\tilde{\eta}$ denote the pre-image $\pr_{\fp}^{-1}(\eta)$ which is invariant under $\theta_k^*$.
For the class $\omega$, we consider the following commutative diagram with exact rows

\begin{equation*}
\begin{tikzcd}
0 \arrow[r] & H^{|\ell|-3}_{\dR}(U_{\ell-4})/ \Fil^{|\ell|-2-s} \arrow[r, "\iota_{\fp} "] \arrow[d, "\theta_\ell^* =0"] & H^{|\ell|-2}_{\fp}( \mathscr{U}_{\ell-4}, |\ell| -2 -s) \arrow[r, "\pr_{\fp}"] \arrow[d, "\theta_\ell^*"] & \Fil^{|\ell|-2-s} H^{|\ell|-2}_{\dR}(U_{\ell-4}) \arrow[r] \arrow[d, "\theta_\ell^*"] \arrow[dl, dashed] &0 \\
0 \arrow[r] & H^{|\ell|-3}_{\dR}(U_{\ell-4})/ \Fil^{|\ell|-2-s} \arrow[r, "\iota_{\fp} "] & H^{|\ell|-2}_{\fp}( \mathscr{U}_{\ell-4}, |\ell| -2 -s) \arrow[r, "\pr_{\fp}"] & \Fil^{|\ell|-2-s} H^{|\ell|-2}_{\dR}(U_{\ell-4}) \arrow[r] &0
\end{tikzcd}.
\end{equation*}
This diagram then implies that one can find a canonical lift $\tilde{\omega} \in \pr_{\fp}^{-1}(\omega)$ such that $\theta_\ell^* \tilde{\omega} = \tilde{\omega}$.
Now, one computes
\begin{align*}
    \AJ_p(\Delta_{\ell, k})(\pi_1^* \omega \cup \pi_2^* \eta) &= \langle  \cl_{\fp}(\Delta_{\ell, k}), \pi_1^* \tilde{\omega} \cup \pi_2^* \tilde{\eta} \rangle_{\fp} \\
    &= \langle (\theta_\ell^*, \theta_k^*) \varphi_* \cl_{\fp}(\mathscr{W}_\gamma), \pi_1^* \tilde{\omega} \cup \pi_2^* \tilde{\eta} \rangle_{\fp} \\
    &= \langle \cl_{\syn}(\mathscr{W}_\gamma), \varphi^* (\pi_1^* \tilde{\omega} \cup \pi_2^* \tilde{\eta}) \rangle_{\fp}\\
    &= \tr_{\mathscr{W}_\gamma, \fp}(\varphi^* (\pi_1^* \tilde{\omega} \cup \pi_2^* \tilde{\eta})).
\end{align*}
Set $\varphi_i := \pi_i \circ \varphi$ for $i = 1, 2$.
The pullback $\varphi_1^* \tilde{\omega} \in H^{|\ell|-2}_{\fp}(\mathscr{W}_\gamma, |\ell|-2-s)$ can be viewed as an element in $H^{|\ell|-3}_{\dR}(W_\gamma)$ via the isomorphism $\iota_{\fp}$, since $\Fil^{|\ell|-2-s} H^{|\ell|-2}_{\dR}(W_\gamma) =0$.
If we set $\Upsilon(\omega) = \iota_{\fp}^{-1}(\varphi_1^* \tilde{\omega})$, then we can rewrite 
\begin{equation}
\begin{split}
    \AJ_p(\Delta_{\ell, k})(\pi_1^* \omega \cup \pi_2^* \eta) &= \tr_{\mathscr{W}_\gamma, \fp}(\varphi_1^* \tilde{\omega} \cup \varphi_2^* \tilde{\eta})) = \tr_{\mathscr{W}_{\gamma, \dR}} (\Upsilon(\omega) \cup \varphi_2^*\eta) \\
    &= \langle \Upsilon(\omega), \varphi_2^*\eta \rangle_{\mathscr{W}_{\gamma, \dR}} = \langle \varphi_{2, *}\Upsilon(\omega), \eta \rangle_{\mathscr{W}_{k-2, \dR}},
\end{split}
\end{equation}
where the second equality is by the compatibility between cup products of finite polynomial cohomology and de Rham cohomology (c.f. \cite[Proposition~2.5]{fp}).
Our goal now is to describe $\varphi_{2, *}\Upsilon(\omega)$ in terms of overconvergent modular forms, which will be done in the next subsection.

\subsection{Computations: the split case}
In this subsection, we consider the case where $p$ splits as $\mathfrak{p}_1 \mathfrak{p}_2$ in the real quadratic field $L$.

To study $\varphi_{2, *}\Upsilon(\omega)$, one needs to fix several polynomials for finite polynomial cohomology. 
Set $T = T_1 T_2$ and
\begin{align*}
    P_1(T_1) &= (1 -\alpha_1 T_1)(1-\beta_1 T_1), \\
    P_2(T_2) &= (1 -\alpha_2 T_2)(1-\beta_2 T_2), \\
    P(T) &= (1 - \alpha_1 \alpha_2 T)(1 - \alpha_1 \beta_2 T)(1 - \beta_1 \alpha_2 T)(1 - \beta_1 \beta_2 T),
\end{align*}
where $\alpha_i, \beta_i$ are the eigenvalues of $T_0({\mathfrak{p}_i})$.
We here refer to \cite[p.1986]{blanco-chacón_fornea_2020} for the explicit choices of the polynomials $a_2(T_1, T_2)$ and $b_1(T_1, T_2)$ such that
\begin{equation*}
    P(T_1T_2) = a_2(T_1, T_2) P_1(T_1) + b_1(T_1, T_2) P_2(T_2).
\end{equation*}
The polynomial $a_2(T_1, T_2)$ is composed of monomials $T_1^x T_2^y$ with $x \leq y$, while $b_1(T_1, T_2)$ is composed of monomials $T_1^x T_2^y$ with $x > y$.
Since $P(T_1 T_2)$ is symmetric in the indices $1$ and $2$,
one can also write
\begin{equation*}
    P(T_1T_2) = a_1(T_1, T_2) P_2(T_2) + b_2(T_1, T_2) P_1(T_1),
\end{equation*}
where $a_1$ (resp. $b_2$) is obtained from $a_2(T_1, T_2)$ (resp. $b_1$) by swapping all the indices.
Finally, we let these polynomials act on the space of overconvergent modular forms by letting $T_i$ act as $V_0(\mathfrak{p}_i)$ (cf. \cite[\S6]{twisted_tripleL_finite_slope}), where $V_0(\mathfrak{p}_i)$ satisfies $U_0(\mathfrak{p}_i)V_0(\mathfrak{p}_i) = \mathrm{id}$. We remark that $V_0(\mathfrak{p}_i) (g\cdot \eta_{\lambda, \mathfrak{d}_L^{-1}}^w) = (V_0(\mathfrak{p}_i)g)\cdot \eta_{\lambda, \mathfrak{d}_L^{-1}}^w$ for any $w \in \Z[\Sigma_L]$ (cf. \cite[Lemma 2.6]{blanco-chacón_fornea_2020}). The readers should note that the definitions of the operator $V_0(\mathfrak{p}_i)$ on either side of this equality are different, and depend on the component $v$ of the weight $(v,n)$, as in the definition of the normalised Hecke operator $T_0(\mathfrak{p}_i)$. Since the roots of the $T_0(\mathfrak{p}_i)$ Hecke polynomial acting on the $g$-isotypic component does not depend on whether we view $g$ as a section of $\omega_{\Sh{A}}^{\ell} \otimes (\wedge^2_{\Sh{O}_L} \Sh{H}_{\Sh{A}})^{-v}$ or of $\omega_{\Sh{A}}^{\ell} \otimes (\wedge^2_{\Sh{O}_L} \Sh{H}_{\Sh{A}})^{-t_L}$ (as $g\cdot \eta_{\lambda, \mathfrak{d}_L^{-1}}^{v-t_L})$, we find that $P(V_0(p))$ kills the differential form $\omega_g$.




Let $g$ be of weight $\ell = (\ell_1, \ell_2)$ as before.
For $j= 1, 2$, there are $\mathfrak{p}_j$-depleted overconvergent cusp forms $g_1^{(j)}, g_2^{(j)}$ such that $g^{[\mathfrak{p}_j]} = P_j (V_0(\mathfrak{p}_j)) g = d_1^{\ell_1-1} g_1^{(j)} + d_2^{\ell_2-1} g_2^{(j)}$ (c.f. \cite[Theorem~4.5]{blanco-chacón_fornea_2020}).
One can then write 
\begin{equation*}
    P(V_0(p)) g = d_1^{\ell_1-1} (h)  + d_1^{\ell_1-1} (h_1) + d_2^{\ell_2 -1} (h_2),
\end{equation*}
where $h = (1 -\alpha_1 \beta_1 \alpha_2 \beta_2 V_0(p)^2 ) d_1^{1-\ell_1} g^{[\mathcal{P}]}$, $h_1 = b_2 g_1^{(1)} + b_1 g_1^{(2)}$ and $h_2 = b_2 g_2^{(1)} + b_1 g_2^{(2)}$ are all overconvergent forms. 
We set $Q(T) = P(p^{|t_L-\ell|} T)$ so that $Q (\Frob_p) g =  P(V_0(p)) g$. 

By \cite[Proposition~5.11]{blanco-chacón_fornea_2020}, one can find nearly overconvergent forms $H, H_1, H_2$ such that 
\begin{equation*}
    \nabla (H) = d_1^{\ell_1 -1} h, \nabla (H_1) = d_1^{\ell_1 -1} h_1, \textrm{ and } \nabla (H_2) = d_2^{\ell_2 -1} h_2.
\end{equation*}
In particular, if we set $G= H+ H_1+ H_2$, then the pair $[ \omega_g, G ]$ represents the class $\tilde{\omega}$ in $H^{|\ell|-2}_{\fp, Q}( \mathscr{U}_{\ell-4}, |\ell| -2 -s)$.

Let us now write down explicitly the expansions of $H, H_1, H_2$. 
Following a similar computation as in \cite[Proposition~5.11]{blanco-chacón_fornea_2020} (and we ignore the terms $w_j$ corresponding to the $\Lambda^2 \mathcal{H}_{\mathcal{A}}$-factor), we have

\begin{align}
    H &= \sum_{i = 0}^{\ell_1 -2} (-1)^i \frac{(\ell_1 -2 )!}{(\ell_1-2 -i)!} d_1^{\ell_1-2 -i} h \times \left ( \omega_{\can, 1}^{\ell_1-2-i} \eta_{\can, 1 }^i \otimes \omega_{\can, 2}^{\ell_2-2} \right ) \otimes \frac{dq_2}{q_2}, \\
    H_1 &= \sum_{i = 0}^{\ell_1 -2} (-1)^i \frac{(\ell_1 -2 )!}{(\ell_1-2 -i)!} d_1^{\ell_1 -2 -i} h_1 \times \left ( \omega_{\can, 1}^{\ell_1-2-i} \eta_{\can, 1 }^i \otimes \omega_{\can, 2}^{\ell_2-2} \right ) \otimes \frac{dq_2}{q_2}, \\
    H_2 &= \sum_{i = 0}^{\ell_2 -2} (-1)^i \frac{(\ell_2 -2 )!}{(\ell_2-2 -i)!} d_2^{\ell_2-2 -i} h_2 \times \left ( \omega_{\can, 1}^{\ell_1-2} \otimes \omega_{\can, 2}^{\ell_2-2-i} \eta_{\can, 2}^i \right ) \otimes \frac{dq_1}{q_1}.
\end{align}

Now, we examine the polynomial $q$-expansion of these forms under $\tau: = \theta_k^* \varphi_{2, *} \varphi_1^* \theta_\ell^*$, which is analogous to the map denoted by $\mathrm{pr}_{r_1}$ in \cite[Remark~3.5]{tripleL_and_GZ_formula_finite_slope}.
As explained in \textit{loc. cit.}, the map $\tau$ undergoes a symmetrization-desymmetrization process, which produces certain factorial numbers.
After a careful computation, we have
\begin{align}
    \tau H &= \sum_{j =s}^{\ell_1-2} (-1)^j j! \binom{\ell_1-2-s}{j-s} \ \zeta^* (d_1^{\ell_1-2 -j} h) \cdot \omega_{\can}^{k-2 - j +s} \eta_{\can}^{j-s} \cdot \frac{dq}{q}, \\
    \tau H_1 &= \sum_{j =s}^{\ell_1-2} (-1)^j j! \binom{\ell_1-2-s}{j-s} \ \zeta^* (d_1^{\ell_1-2 -j} h_1) \cdot \omega_{\can}^{k-2 - j +s} \eta_{\can}^{j-s} \cdot \frac{dq}{q}, \\
    \tau H_2 &= \sum_{j =s}^{\ell_2-2} (-1)^j j! \binom{\ell_2-2-s}{j-s} \ \zeta^* (d_2^{\ell_2-2 -j} h_2) \cdot \omega_{\can}^{k-2 - j +s} \eta_{\can}^{j-s} \cdot \frac{dq}{q}.
\end{align}
Under the unit root splitting, one is left with only the first term $j=s$, which recovers the result of \cite[Proposition~5.12]{blanco-chacón_fornea_2020}.
The weight $k$ form $\tau(G) = \tau(H + H_1 + H_2)$ will play a similar role as $\mathrm{pr}_{r_1} (G^{[p]} \times h)$ in the case of triple product $p$-adic $L$-functions in \cite[\S~3]{tripleL_and_GZ_formula_finite_slope}.

By the short exact sequence of finite polynomial cohomology,
$\varphi_{2, *} \Upsilon(\omega)$ and $\tau(G)$ are related by the identity $Q(\mathrm{Frob}_p) (\varphi_{2, *} \Upsilon(\omega)) = \tau(G)$.
Consequently, the Abel--Jacobi image $\AJ_p(\Delta_{\ell, k})(\pi_1^* \omega \cup \pi_2^* \eta) $ can be expressed in terms of $\tau (G)$.
Namely, one has 
\begin{align*}
    \langle Q(\mathrm{Frob}_p) \varphi_{2, *} \Upsilon(\omega), \eta \rangle_{\dR} &= \langle \tau (G), \eta \rangle_{\dR} = \langle  e^{\leq a} \tau (G), \eta \rangle_{\dR} \\
    & = \frac{\langle e^{\leq a} (\tau H + \tau H_1 + \tau H_2), f^* \rangle}{\langle f^*, f^*\rangle}, 
\end{align*}
where the second equality follows from the fact that $\eta$ lies in the slope $\leq a$-part of the Frobenius and \cite[Lemma~5.3]{tripleL_and_GZ_formula_finite_slope}.

Now we examine the components $e^{\leq a } \tau H_1$ and $e^{\leq a } \tau H_2$.
Recall from \cite[Lemma 7.22]{twisted_tripleL_finite_slope} that we have the equation
\begin{equation*}
    U \zeta^* ( V_0(\mathfrak{p_2} )g^{[\mathfrak{p}_1]} ) = 0 = U \zeta^* ( V_0(\mathfrak{p_1} )g^{[\mathfrak{p}_2]} )
\end{equation*}
for any overconvergent form $g$, which is derived essentially from examining the $q$-expansion.
Recall also that $h_1 = b_2 g_1^{(1)} + b_1 g_1^{(2)}$.
As $g_1^{(1)}$ is $\mathfrak{p}_1$-depleted and $b_2$ can be written as a polynomial in $V_0(p)$ and $V_0(\mathfrak{p}_2)$ which is divisible by $V_0(\mathfrak{p}_2)$, we see that 
$U \zeta^* (b_2 g_1^{(1)}) = 0$ and similarly $U \zeta^* (b_1 g_1^{(2)}) = 0$.
Furthermore, the slope projector $e^{\leq a}$ can be written as a convergent power series in $U$ with no constant term, so we have 
\begin{equation*}
    e^{\leq a} \zeta^* ( b_2 g_1^{(1)}) =  e^{\leq a} \zeta^* (b_1 g_1^{(2)}) = e^{\leq a} \zeta^* (h_1) = 0.
\end{equation*}
In fact, $e^{\leq a} \zeta^* ( d_1^n h_1) =0$ for all $n \in \mathbb{Z}$ since $d_1^n g_1^{(i)}$ is still $\mathfrak{p}_i$-depleted.
A similar argument also holds for $ e^{\leq a} \zeta^* ( d_2^n h_2) $.
As a consequence, one gets 
\begin{equation*}
    e^{\leq a} (\tau H_1 + \tau H_2)=0.
\end{equation*}

Now, we only need to focus on $\tau H$. 
The computation of the Euler factors is identical to that of \cite[Theorem~5.14]{blanco-chacón_fornea_2020}, except that we are using different normalisations of Hecke operators.
In conclusion, we have the following result.

\begin{theorem}
    Suppose $p$ splits in $L$. We have
    \begin{equation}
        \AJ_p(\Delta_{\ell, k}) (\pi_1^* \omega_g \cup \pi_2^* \eta_f ) = \frac{\mathscr{E}_{0,p}(g,f^*) }{\mathscr{E}_p(g,f^*)} \cdot \frac{\langle e^{\leq a}  H', f^*  \rangle }{ \langle f^*, f^* \rangle },
    \end{equation}
    where $H' =  \sum_{j =s}^{\ell_1-2} (-1)^j j! \binom{\ell_1-2-s}{j-s} \ \zeta^* (d_1^{-1 -j} g^{[\mathcal{P}]}) \cdot \omega_{\can}^{k-2 - j +s} \eta_{\can}^{j-s} \cdot \frac{dq}{q}$ is viewed as a nearly overconvergent form, 
    $\mathscr{E}_{0,p}(g, f^*) = 1 -\alpha_1\beta_1 \alpha_2 \beta_2 ({(\alpha_{f}^{*})}^{-1} p^{-s-1})^2 $, and $\mathscr{E}_p(g, f^*) = P({(\alpha_{f}^{*})}^{-1} p^{-s-1})$ are as in Theorem \ref{Theorem: interpolation formula} with $t$ replaced by (-s-1,0).
\end{theorem}

Now, one compares the form $H'$ to $\zeta^*( \nabla^{(-(s+1), 0)} g_{\ell}^{[\mathcal{P}]})$.
Recall that the polynomial $q$-expansion of $\zeta^*( \nabla^{(-(s+1), 0)} g_{\ell}^{[\mathcal{P}]})$ is 
\begin{equation}
    \sum_{j=0}^{\ell_1-s-2} \binom{-t}{j} \prod_{i =0}^{j-1} (\ell_1 -s -2 -i) \zeta^* \left (d_1^{-s-1-j} g_{\ell}^{[\mathcal{P}]}(q) \right ) \cdot \omega_{\can}^{k -j} \eta_{\can}^{j}.
\end{equation}
After a change of variable, one can rewrite $H'$ as
\begin{equation}
    H' = \sum_{j =0}^{\ell_1-s -2} (-1)^{j+s} (j+s)! \binom{\ell_1-2-s}{j} \ \zeta^* (d_1^{-1 -s-j} g^{[\mathcal{P}]}) \cdot \omega_{\can}^{k-2 - j} \eta_{\can}^{j} \cdot \frac{dq}{q}.
\end{equation}
From the above two descriptions and formula (\ref{equation: overconvergent projection}), it can be derived from a careful computation that 
\begin{equation}
    H^\dagger (H') = (-1)^s s!  H^\dagger (\zeta^*( \nabla^{(-(s+1), 0)} g_{\ell}^{[\mathcal{P}]})) .
\end{equation}
The $p$-adic Gross--Zagier formula now follows easily, with an extra Euler factor 
$$\mathscr{E}(f^*) = (1 - \frac{\beta^*_{f}}{\alpha^*_{f}} ) = (1 -\beta_{f}^2 \chi_f^{-1}(p) p^{1-k})$$
coming from comparing Petersson products with and without level at $p$ (c.f. \cite[\S~3.3]{tripleL_and_GZ_formula_finite_slope}).

In summary, we obtain
\begin{theorem} \label{theorem: GZ formula split case}
    Suppose $p$ is split in $L$.
    Let $(P, Q) = ( (v, n), (w, m))$ be a classical point corresponding to a balanced tuple $(\ell, k)$ that satisfies Assumption \ref{assumption: classical specialisation}. 
    Then we have 
\begin{equation}
    \mathscr{L}_p(\omega_\mathbf{g}, \omega_\mathbf{f})(P, Q) = \frac{(-1)^s}{s! \mathscr{E}(f_k^*)} \frac{\mathscr{E}_p(g_\ell, f_k^*)}{\mathscr{E}_{0,p} (g_\ell, f_k^*) } \AJ_p (\Delta_{\ell, k}) (\pi_1^* \omega_{g_\ell} \cup \pi_2^* \eta_{f_k}).
\end{equation}
\end{theorem}


    

\subsection{The inert case}

In this section, we will examine the case where $p$ is inert in the quadratic extension $L$, which is not covered in \cite{blanco-chacón_fornea_2020}.
We will keep most of the notations as before.

In this situation, there is only one Hecke operator $V_0(p)$ and one Hecke polynomial 
$$P(T) = (1 -\alpha T)(1- \beta T),$$
with the usual relation $P(V_0(p)) \omega_g = P(p^{|t_L-\ell|} \mathrm{Frob}_p) \omega_g = \omega_g^{[p]}$.
We let $G = \nabla_1^{-1} \omega_g^{[p]}$ and $Q(T) = P(p^{|t_L-\ell |}T)$.
Then the pair $[\omega_g , G]$ represents the class $\tilde{\omega}$ in $H^{|\ell|-2}_{\fp, Q}( \mathscr{U}_{\ell-4}, |\ell| -2 -s)$.

The polynomial $q$-expansion of $G$ is 
\begin{equation}
    G = \sum_{j = 0}^{\ell_1 -2} (-1)^j \frac{(\ell_1 -2 )!}{(\ell_1-2 -j)!} d_1^{-1- j} g^{[p]} \times \left ( \omega_{\can, 1}^{\ell_1-2-i} \eta_{\can, 1 }^i \otimes \omega_{\can, 2}^{\ell_2-2}  \right ) \otimes \frac{dq_2}{q_2}.
\end{equation}
Similarly, one computes $\tau G$ and get
\begin{equation}
    \tau G = \sum_{j =s}^{\ell_1-2} (-1)^j j! \binom{\ell_1-2-s}{j-s} \ \zeta^* ( d_1^{-1- j} g^{[p]}) \cdot \omega_{\can}^{k-2 - j +s} \eta_{\can}^{j-s} \cdot \frac{dq}{q}.
\end{equation}
Again, we have the relation $Q(\mathrm{Frob}) (\varphi_{2, *} \Upsilon(\omega)) = \tau G$ and the following theorem.
\begin{theorem} 
    Suppose $p$ is inert in $L$. We have
    \begin{equation}
        \AJ_p(\Delta_{\ell, k}) (\pi_1^* \omega_g \cup \pi_2^* \eta_f ) = \frac{1}{\mathscr{E}_p(g, f^*)} \cdot \frac{ \langle e^{\leq a}  \tau G , f^* \rangle}{\langle f^*, f^* \rangle},
    \end{equation}
    where 
    $\mathscr{E}_p(g, f^*) = P((\alpha_{f}^*)^{-1} p^{-s-1})$ is as in Theorem \ref{Theorem: interpolation formula}. 
\end{theorem}

With a change of variables, one can rewrite
\begin{equation}
    \tau G = \sum_{j =0}^{\ell_1-s -2} (-1)^{j+s} (j+s)! \binom{\ell_1-2-s}{j} \ \zeta^* (d_1^{-1 -s-j} g^{[p]}) \cdot \omega_{\can}^{k-2 - j} \eta_{\can}^{j} \cdot \frac{dq}{q}.
\end{equation}
A similar computation shows that  
$$\tau G = (-1)^s s! \zeta^* ( \nabla^{(-(s+1), 0)} g_{\ell}^{[p]}) .$$
Consequently, we obtain the $p$-adic Gross--Zagier formula:
\begin{theorem} \label{theorem: GZ formula inert case}
    Suppose $p$ is inert in $L$.
    Let $(P, Q) = ( (v, n), (w, m))$ be a classical point corresponding to a balanced tuple $(\ell, k)$ that satisfies Assumption \ref{assumption: classical specialisation}. 
    Then we have 
    \begin{equation}
    \mathscr{L}_p(\omega_{\mathbf{g}}, \omega_{\mathbf{f}})(P, Q) = \frac{(-1)^s}{s! \mathscr{E}(f_k^*)} {\mathscr{E}_p(g_\ell, f_k^*)} \AJ_p (\Delta_{\ell, k}) (\pi_1^* \omega_{g_\ell} \cup \pi_2^* \eta_{f_k}),
    \end{equation}
    where $\mathscr{E}_p(g_\ell, f_k^*) =  P(\alpha_{f^*}^{-1} p^{-s-1})$ and $\mathscr{E}(f_k^*) =  (1 - \frac{\beta^*_{f}}{\alpha^*_{f}} ) $ are as before.
\end{theorem}

\printbibliography
\end{document}